\documentclass{qam-l}

\usepackage{graphicx}

\newtheorem{theorem}{Theorem}[section]
\newtheorem{lemma}[theorem]{Lemma}
\newtheorem{corollary}[theorem]{Corollary}

\theoremstyle{definition}

\theoremstyle{remark}
\newtheorem{remark}[theorem]{Remark}

\numberwithin{equation}{section}

\newcommand{\numberthis}{\addtocounter{equation}{1}\tag{\theequation}}
\newcommand{\ds}{\displaystyle}
\newcommand{\dd}{\, \mathrm{d}}
\newcommand{\ii}{\mathrm{i}}
\newcommand{\ee}{\mathrm{e}}
\newcommand{\Vhat}{\widehat{V}}
\newcommand{\rhohat}{\widehat{\rho}}

\newcommand{\g}{g}
\newcommand{\gk}{g}
\newcommand{\Pset}{\mathcal{P}}
\newcommand{\eqand}{\quad \text{and} \quad}

\newcommand{\myspace}{\\[0.4cm]}

\DeclareMathOperator{\supp}{supp}

\begin{document}

\title[Sensitivity of Anomalous Localized Resonance Phenomena]%
{Sensitivity of anomalous localized resonance phenomena with respect to
dissipation}


\author[T. Meklachi]{Taoufik Meklachi}
\address{%
    University of Houston,
    Department of Mathematics, 651 PGH
    Houston, Texas, 770204-3008}
\email{tmclachi@math.uh.edu}
\thanks{The work of Taoufik Meklachi was supported by the 
        Air Force through grant AFOSR YIP Early Career Award FA9550-13-1-0078}

\author[G. W. Milton]{Graeme W. Milton}
\address{%
    University of Utah,
    Department of Mathematics,
    155 South 1400 East, Room 233,
    Salt Lake City, Utah, 84112-0090}
\email{milton@math.utah.edu}
\thanks{The work of Graeme W. Milton was supported by the National
        Science Foundation through grant DMS-1211359.}

\author[D. Onofrei]{Daniel Onofrei}
\address{%
    University of Houston,
    Department of Mathematics, 651 PGH,
    Houston, Texas, 770204-3008}
\email{onofrei@math.uh.edu}
\thanks{The work of Daniel Onofrei was supported under the Simons
        Collaborative Grant and the Air Force through grant AFOSR YIP Early 
        Career Award FA9550-13-1-0078}

\author[A. E. Thaler]{Andrew E. Thaler}
\address{%
    University of Utah,
    Department of Mathematics,
    155 South 1400 East, Room 233,
    Salt Lake City, Utah, 84112-0090}
\email{thaler@math.utah.edu}
\thanks{The work of Andrew E. Thaler was supported by the National
        Science Foundation through grant DMS-1211359.}

\author[G. Funchess]{Gregory Funchess}
\address{%
    University of Houston, 
    Department of Mathematics, 651 PGH,
    Houston, Texas, 770204-3008}
\email{gfunchess@gmail.com}
\thanks{The work of Gregory Funchess was supported by the 
        Air Force through grant AFOSR YIP Early Career Award FA9550-13-1-0078}


\date{}


\begin{abstract}
    We analyze cloaking due to anomalous localized resonance in the
    quasistatic regime in the case when a general charge density
    distribution is brought near a slab superlens. If the charge density
    distribution is within a critical distance of the slab, then the
    power dissipation within the slab blows up as certain electrical
    dissipation parameters go to zero. The potential remains bounded far
    away from the slab in this limit, which leads to cloaking due to
    anomalous localized resonance. On the other hand, if the charge
    density distribution is further than this critical distance from the
    slab, then the power dissipation within the slab remains bounded and
    cloaking due to anomalous localized resonance does not occur. The
    critical distance is shown to strongly depend on the the rate at
    which the dissipation outside of the slab goes to zero.  
\end{abstract}

\maketitle


\section{Introduction}\label{sec:introduction}

In this paper, we discuss anomalous localized resonance phenomena
observed at the interface between positive-index and negative-index
materials. Such phenomena have been at the center of an interesting
cloaking strategy \cite{Milton:2006:CEA, Bruno:2007:SCS,
Milton:2007:OPL, Nicorovici:2007:QCT, Milton:2008:SFG,
Nicorovici:2008:FWC, Nicorovici:2009:CPR, Bouchitte:2010:CSO,
Nicorovici:2011:RLD, Xiao:2012:TEC, Ammari:2013:ALR, Ammari:2013:STN,
Ammari:2013:STNII, Nguyen:2013:SUC, Bergman:2014:PIP, Thaler:2014:BVI}.

As illustrated in Figure~\ref{fig:slab}, the (2D) geometry we consider
consists of a central layer in $\mathcal{S} \equiv [0,a] \times
(-\infty, +\infty)$ bordered by a layer to the left in $\mathcal{C}
\equiv (-\infty,0) \times (-\infty,+\infty)$ and a layer to the right in
$\mathcal{M} \equiv (a,+\infty) \times (-\infty, +\infty)$. We work in
the nonmagnetic quasistatic regime, i.e., the regime in which the
magnetic permeability equals 1 and relevant wavelengths and attenuation
lengths are much larger than other dimensions in the problem (such as
$a$, the thickness of the slab $\mathcal{S}$). In this regime the
complex electric potential $V$ satisfies the Laplace equation 
\begin{equation}\label{eq:V_pde}
        -\nabla\cdot\left[\varepsilon(x,y)\nabla V(x,y)\right] = \rho 
            \quad \text{in} \ \mathbb{R}^2, \\ 
\end{equation}
where $\varepsilon$ is the dielectric constant (relative permittivity)
and $\rho$ is a given charge density distribution. (The potential $V$ is
also subject to certain continuity conditions and conditions at infinity
--- these are discussed in Section~\ref{sec:derivation_V}.) We assume
that the charge density distribution $\rho$ is real valued; we also take
$\rho \in \Pset$, where
\begin{equation}\label{eq:P_def}
    \Pset \equiv \{\rho \in L^2(\mathcal{M}) \cap
        L^{\infty}(\mathcal{M}) : \rho \ \text{has compact support in}
        \ \mathcal{M}\}.
\end{equation}
Throughout this paper we also assume
\begin{equation}\label{eq:supp_rho}
    0 < |\supp \rho| < \infty,
\end{equation}
where $|U|$ denotes the Lebesgue measure of the set $U$. Note that this
restriction on the support of $\rho$ excludes dipolar sources.

For the purposes of the current paper we assume the layers are occupied
by three different materials such that the imaginary parts of their
dielectric constants are small (corresponding to small losses) and the
real parts of their dielectric constants are equal but with opposite
signs. In particular we take the dielectric constant $\varepsilon(x,y)$
to be
\begin{equation}\label{eq:dc}
    \varepsilon(x,y) \equiv 
    \begin{cases} 
        \varepsilon_c = 1+\ii\mu & \text{if } x < 0, \\
        \varepsilon_s = -1+\ii\delta & \text{if } 0 \le x \le a, \\
        \varepsilon_m = 1 & \text{if } x > a,
    \end{cases}
\end{equation}
where $0 < \delta < 1$ and $\mu = \delta+\lambda\delta^{\beta}$ for some
constants $\lambda \in \mathbb{R}$ and $\beta > 0$. In the limit $\delta
\rightarrow 0^+$ the moduli \eqref{eq:dc} are that of a quasistatic
two-dimensional superlens (``poor man's superlens''). The question we
address in this paper is to determine those $\rho$ for which the power
dissipation in this superlens blows up as $\delta \rightarrow 0^+$. As
we shall explain shortly this is closely tied with cloaking due to
anomalous resonance. Curiously we will see that the answer depends on
the value of $\beta$, thus showing the sensitivity of the energy
dissipation rate to perturbations. 

We will say that $\lambda$ is
\emph{feasible} if 
\begin{equation}\label{eq:feasible_lambda}
    \lambda > 0  \quad \text{for} \ 0 < \beta < 1, \qquad 
    \lambda \ge -1 \quad \text{for} \ \beta = 1,  \quad \text{or} \quad 
    \lambda \ne 0 \quad \text{for} \ \beta > 1.
\end{equation}
We define $0 < \delta_{\mu}(\beta,\lambda) < 1$ such that $\mu \ge 0$
for $0 < \delta \le \delta_{\mu}$ (which is required physically
\cite{Milton:2005:PSQ} --- the restrictions we placed on $\lambda$
ensure that such a $\delta_{\mu}$ exists). Note that the materials to
the left and right of the slab are both vacuum if $\beta = 1$ and
$\lambda = -1$.    
\begin{figure}[!bt]
    \begin{center}
        \includegraphics{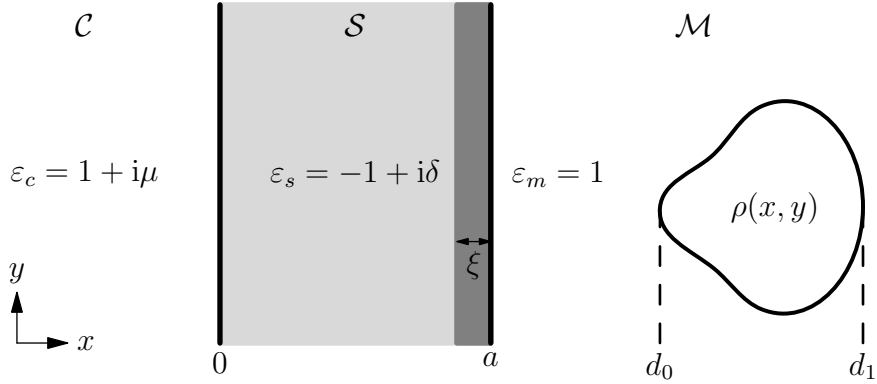}
    \end{center}
    \caption{We consider a slab geometry with a dielectric
    constant as illustrated in the figure --- the slab (shaded
    light gray) is in the region $\mathcal{S} = [0,a] \times
    (-\infty,+\infty)$. The charge density $\rho$ has compact
    support in the region $x > a$. For certain charge densities
    $\rho$ that are close enough to $a$, the energy dissipation
    in the slab (in particular in the darkly shaded region
    $a-\xi < x < a$) tends to infinity as a sequence $\delta_j$
    tends to 0.}
    \label{fig:slab}
\end{figure}
Given a charge density $\rho(x,y)\in \Pset$ with compact support in
$\mathcal{M}$, we define
\begin{equation}\label{eq:d0_d1}
    d_0 \equiv \min\{x : (x,y) \in \supp\rho\} \eqand 
    d_1 \equiv \max\{x : (x,y) \in \supp\rho\}
\end{equation}
(see Figure~\ref{fig:slab}).
Since $\rho$ has compact support in $\mathcal{M}$, we have
\begin{equation}\label{eq:h0_h1}
    \supp \rho \subseteq [d_0,d_1] \times [h_0,h_1]
\end{equation}
for some constants $h_0 < h_1$. In order to enforce charge conservation,
we require 
\begin{equation}\label{eq:zero_charge}
    \ds\int_{d_0}^{d_1}\ds\int_{h_0}^{h_1} \rho(x,y) \dd y \dd x = 0.
\end{equation}
The physical charge density is $\Re(\rho \ee^{-\ii \omega t})$ and the
physical time-harmonic electric field is given by $\mathbf{E} =
\Re\left(-\nabla V \ee^{-\ii \omega t}\right)$.  

We say anomalous localized resonance (ALR) occurs if the following two
properties hold as $\delta \rightarrow 0^+$ \cite{Milton:2005:PSQ}:
\begin{enumerate}
    \item $|V| \rightarrow \infty$ in certain localized regions with
          boundaries that are not defined by discontinuities in the
          relative permittivity and 
    \item $V$ approaches a smooth limit outside these localized regions. 
\end{enumerate}
For example, when $\rho$ is a dipole, $\varepsilon_c = \varepsilon_m =
1$, and when ALR occurs, as the loss in the lens (represented by
$\delta$) tends to zero the potential diverges and oscillates wildly in
regions that contain the boundaries of the lens. It is important to note
that the boundaries of the resonant regions move as the dipole is moved.
Outside the resonant regions the potential converges to what we expect
from perfect lensing \cite{Pendry:2000:NRM, Pendry:2002:NFL}. This
behavior and its relation to subwavelength resolution in imaging
(superlensing) were first discovered by Nicorovici, McPhedran, and
Milton \cite{Nicorovici:1994:ODP} and were analyzed in more depth by
Milton, Nicorovici, McPhedran, and Podolskiy \cite{Milton:2005:PSQ}.  

Milton, Nicorovici, McPhedran, and Podolskiy \cite{Milton:2005:PSQ}
showed that if $\rho$ is a dipole and $\varepsilon_c = \varepsilon_m =
1$, then ALR occurs if $a < d_0 < 2a$, where $d_0$ is the location of
the dipole. In this case there are two locally resonant strips --- one
centered on each face of the slab. As mentioned above, outside these
regions the potential converges to a smooth function that satisfies
mirroring properties of a perfect lens. In particular, to an observer
far enough to the right of the lens it will appear only as if there is a
dipole at $d_0$; to an observer far enough to the left of the lens it
will appear only as if there is a dipole located at $-d_0$ 
\cite{Milton:2005:PSQ}. In neither case can the observer determine
whether or not a lens is present. (However, if either observer is close
to the lens, the presence of the lens will be obvious due to the
resonance.) If $d_0 > 2a$, then there is no resonance and again the
potential converges to a smooth function that satisfies the mirroring
properties expected of a perfect lens. That is, to an observer far
enough to the right of the lens (beyond the dipole) it will appear as if
there is a dipole at $d_0$ and no lens, while to an observer to the left
of the lens it will appear as if there is a dipole at $d_0-a$ and no
lens \cite{Pendry:2000:NRM, Milton:2005:PSQ, Yan:2008:CSC}.

Cloaking due to ALR (CALR) can be understood from an energetic
perspective. First, consider the quantity
\begin{equation}\label{eq:intro_pd}
    E(\delta) \equiv \delta\int_{0}^{a}\int_{-\infty}^{\infty} 
        |\nabla V|^2 \dd y \dd x;
\end{equation}
$E(\delta)$ is proportional to the time-averaged electrical power
dissipated in the slab. Suppose $\rho$ is independent of $\delta$ such
that, in the limit $\delta \rightarrow 0^+$, we have $E(\delta)
\rightarrow \infty$ and $|V|/\sqrt{E(\delta)} \rightarrow 0$ for all
$(x,y) \in \mathbb{R}^2$ with $|x| > b$ for some $b > 0$. This blow-up
in the power dissipation is not physical, as it implies the fixed source
$\rho$ must produce an infinite amount of power in the limit $\delta
\rightarrow 0^+$ \cite{Milton:2006:CEA, Ammari:2013:STN}.
The power dissipation was proved to blow up as $\delta \rightarrow 0^+$
for finite collections of dipolar sources close enough to the slab by 
Milton et al.\ \cite{Milton:2005:PSQ, Milton:2006:CEA}; 
see also the work of Bergman \cite{Bergman:2014:PIP}.

To make sense out of this we rescale the source $\rho$ by defining
$\rho_r \equiv \rho/\sqrt{E(\delta)}$. Since \eqref{eq:V_pde} is linear,
the associated potential will be $V_r \equiv V/\sqrt{E(\delta)}$ and,
thanks to \eqref{eq:intro_pd}, the rescaled time-averaged electrical
power dissipation will be 
\[
    E_r(\delta) \equiv\delta\int_{0}^{a}\int_{-\infty}^{\infty} 
        |\nabla V_r|^2 \dd y \dd x =
        \delta\int_{0}^{a}\int_{-\infty}^{\infty} 
        \frac{|\nabla V|^2}{E(\delta)} \dd y \dd x = 1.
\] 
Thus the source $\rho_r$ produces constant power
independent of $\delta$. Also, the rescaled potential satisfies $|V_r| =
|V|/\sqrt{E(\delta)} \rightarrow 0$ as $\delta \rightarrow 0^+$ for $|x|
> b$, implying that the source $\rho_r$ becomes invisible in this limit
to observers beyond $|x| = b$. This idea was introduced by Milton and
Nicorovici \cite{Milton:2006:CEA}; also see the work by Kohn and 
Vogelius \cite{Kohn:2012:VPC} and the works by Ammari et al.\ 
\cite{Ammari:2013:STN, Ammari:2013:STNII}.

Cloaking due to anomalous localized resonance in the quasistatic regime
was first analyzed by Milton and Nicorovici \cite{Milton:2006:CEA}. They
used separation of variables and rigorous analytic estimates to prove
that if $\varepsilon_c = \varepsilon_m = 1$ and a fixed field is applied
to the system (e.g., a uniform field at infinity), then a polarizable
dipole located in the region $a < d_0 < 3a/2$ causes anomalous localized
resonance and is cloaked in the limit $\delta \rightarrow 0^+$; if
$\varepsilon_c \ne \varepsilon_m = 1$ (here $\varepsilon_c$ has no
relation to the value we chose in \eqref{eq:dc}), then the cloaking
region becomes $a < d_0 < 2a$.  

Milton and Nicorovici \cite{Milton:2006:CEA} also derived analogous
results for circular cylindrical lenses. In that case they assumed the
relative permittivity was $\varepsilon_c$ for $0 < r < r_c$,
$\varepsilon_s = -1 + \ii\delta$ for $r_c < r < r_s$, and $\varepsilon_m
= 1$ for $r_s < r$. With $r_0$ denoting the distance of the polarizable
dipole from the origin, the cloaking region was found to be $r_s < r_0 <
r_* = r_s^2/r_c$ if $\varepsilon_c \ne \varepsilon_m$ and $r_s < r_0 <
r_{\#} = \sqrt{r_s^3/r_c}$ if $\varepsilon_c = \varepsilon_m$. In
particular they proved that an arbitrary number of polarizable dipoles
within the cloaking region will be cloaked --- Nicorovici, Milton,
McPhedran, and Botten provided numerical verification of this result
\cite{Nicorovici:2007:QCT}. Milton and Nicorovici \cite{Milton:2006:CEA}
also extended their results to the finite-frequency and
three-dimensional cases for the Veselago slab lens
\cite{Veselago:1967:ESS} (where $\varepsilon_c = \varepsilon_m = 1$).

To summarize, suppose $\varepsilon_c = \varepsilon_m = 1$ and the
polarizable dipole is absent and a uniform electric field at infinity is
applied to the slab lens configuration. The lens will not perturb this
external field in the limit $\delta \rightarrow 0^+$, and, hence, is
invisible to external observers \cite{Nicorovici:1994:ODP,
Milton:2005:PSQ}. When the polarizable dipole is placed in this uniform
field but outside of the cloaking region (so $d_0 > 3a/2$), it will
become polarized and create a dipole field of its own which interacts
with the lens. If $d_0 > 2a$ as well there will be no resonance in the
limit $\delta \rightarrow 0^+$; to an external observer, the lens will
be invisible but the dipole will be clearly visible in this limit. If
$3a/2 < d_0 < 2a$, resonance will occur as $\delta \rightarrow 0^+$ but
it will be localized to strips around the boundaries of the lens --- in
particular the resonant fields will not interact with the dipole. The
dipole will still be visible in this limit but to an observer outside of
the resonance region (and outside the lens) the lens will be invisible.
Finally, if $d_0 < 3a/2$ (so the polarizable dipole is within the
cloaking region), the resonant field will interact with the polarizable
dipole and effectively cancel the effect of the external field on it. In
other words, the net field at the location of the polarizable dipole
will be zero, and, hence, its induced dipole moment will be zero (in the
limit as $\delta \rightarrow 0^+$) --- both the lens and the dipole will
be invisible to external observers. See Figure 3 in the paper by Milton
and Nicorovici \cite{Milton:2006:CEA} and the figures in the work by
Nicorovici, Milton, McPhedran, and Botten \cite{Nicorovici:2007:QCT} for
dramatic illustrations of this in the circular cylindrical case.

Nicorovici, McPhedran, Enoch, and Tayeb studied CALR for the circular
cylindrical superlens in the finite-frequency case
\cite{Nicorovici:2008:FWC}. For physically plausible values of $\delta$
they discovered that the cloaking device (the superlens) can effectively
cloak a tiny cylindrical inclusion located within the cloaking region
but that the superlens does not necessarily cloak itself --- they deemed
this phenomenon the ``ostrich effect.'' In the quasistatic (long
wavelength) limit, however, the lens can effectively cloak both the
inclusion and itself even at rather large values of $\delta$, which was
also pointed out in the case of a polarizable dipole
\cite{Milton:2006:CEA}.  

Bouchitt\'e and Schweizer \cite{Bouchitte:2010:CSO} considered an
annular lens with inner and outer radii of 1 and $R$, respectively, and
relative permittivity $\varepsilon_s = -1 + \ii \delta$ embedded in
vacuum. They proved that a small circular inclusion of radius
$\gamma(\delta)$ (with $\gamma(\delta) \rightarrow 0$ as $\delta
\rightarrow 0^+$) is cloaked in the limit $\delta \rightarrow 0^+$ if it
is located within the annulus $R < |x_0| < R_* = R^{3/2}$, where $x_0$
is the position of the circular source. If $|x_0| > R_*$, then the
source is visible but the annular superlens is not. Both of these
results are consistent with the results of Milton and Nicorovici
\cite{Milton:2006:CEA}. Bruno and Lintner \cite{Bruno:2007:SCS}
considered a similar scenario, where they showed numerically that a
small dielectric disk is not perfectly cloaked. They verified
(numerically) that an annular superlens embedded in vacuum by itself is
invisible to an external applied field in the zero loss limit (assuming
the source is at a position further than $R_*$ from the origin) --- a
fact that was first shown analytically by Nicorovici, McPhedran, and
Milton \cite{Nicorovici:1994:ODP}; however, they also showed that
elliptical superlenses can cloak polarizable dipoles that are near
enough to the lens but that such lenses are not invisible themselves.
That is, the polarizable dipole is cloaked but it is obvious to external
observers that something is being hidden --- this is another example of
the ``ostrich effect'' introduced by Nicorovici et al.\
\cite{Nicorovici:2008:FWC}.  

Kohn, Lu, Schweizer, and Weinstein used variational principles to derive
resonance results in the quasistatic regime in core/shell geometries
(where the superlens resides in the shell) that are not necessarily
radial \cite{Kohn:2012:VPC}. They assumed the source was supported on
the boundary of a disk in $\mathbb{R}^2$, and obtained results similar
to those described above.  

Ammari, Ciraolo, Kang, Lee, and Milton \cite{Ammari:2013:STN,
Ammari:2013:STNII} used properties of certain Neumann--Poincar\'e
operators to prove results analogous to those of Milton and Nicorovici
\cite{Milton:2006:CEA}. The most general results they derived hold for
very general core/shell geometries and charge density distributions
$\rho$ with compact support in the quasistatic regime. In the circular
cylindrical case their requirements are more explicit and involve gap
conditions on the Fourier coefficients of the Newtonian potential of
$\rho$. Although these gap conditions may be difficult to deal with for
a given source, they verified that their results are consistent with
those of Milton and Nicorovici \cite{Milton:2006:CEA} when $\rho$ is a
dipole or quadrupole. Their results can be summarized as follows. First,
if the support of $\rho$ is completely contained within the cloaking
region ($r_s < r_0 < r_*$ if $\varepsilon_c \ne \varepsilon_m = 1$ and
$r_s < r_0 < r_{\#}$ if $\varepsilon_c = \varepsilon_m = 1$), and if
$\rho$ satisfies the gap property, then CALR occurs. Second, weak CALR
(defined by $\limsup_{\delta \rightarrow 0^+} E(\delta) = \infty$ and
$|V| < C$ for all $\delta$ where $C>0$ is independent of $\delta$)
occurs if the support of $\rho$ is completely inside the cloaking region
and the Newtonian potential does not extend harmonically to all of
$\mathbb{R}^2$. Third, if $\Re(\varepsilon_s) \ne -1$, then CALR does
not occur. Fourth, CALR does not occur for any isotropic constant values
of $\varepsilon_c$ and $\varepsilon_s$ when the core and shell are
concentric spheres in $\mathbb{R}^3$. Using a folded geometry approach
(extending that of Leonhardt and Philbin \cite{Leonhardt:2006:GRE} and
Leonhardt and Tyc \cite{Leonhardt:2009:BIN}), Ammari, Ciraolo, Kang,
Lee, and Milton \cite{Ammari:2013:ALR} proved that CALR can occur in 3D
when the core and shell are concentric spheres and the shell has a
certain anisotropic relative permittivity --- see the work of Milton,
Nicorovici, McPhedran, Cherednichenko, and Jacob \cite{Milton:2008:SFG}
for the analogous problem in 2D.  

Nicorovici, McPhedran, Botten, and Milton \cite{Nicorovici:2009:CPR}
asked whether or not one can enlarge the cloaking region by spatially
overlapping the cloaking regions of identical circular cylindrical
superlenses. Curiously they found that doing so \emph{reduces} the
cloaking effect (at least in the quasistatic regime). The cloaking
region can be extended by arranging the disks in such a way that their
corresponding cloaking regions just touch.

Milton, Nicorovici, and McPhedran \cite{Milton:2007:OPL} utilized a
correspondence (first discovered although not fully exploited by
Yaghjian and Hansen \cite{Yaghjian:2006:PWS}) between the perfect
Veselago lens at a fixed frequency in the long-time limit and the lossy
Veselago lens in the quasistatic limit to show that transverse magnetic
dipole sources that generate bounded power eventually become cloaked if
they are within the cloaking region ($a < d_0 < 3a/2$). Xiao, Huang,
Dong, and Chan obtained similar results in the case when both the
permittivity and permeability of the Veselago lens have a positive
imaginary part \cite{Xiao:2012:TEC}.  

Finally, Nguyen proved that arbitrary inhomogeneous objects are
magnified by properly constructed superlenses in both the quasistatic
and finite-frequency regimes in two and three dimensions
\cite{Nguyen:2013:SUC}.  

In this paper we consider the scenario sketched in Figure~\ref{fig:slab}
and described by \eqref{eq:V_pde}--\eqref{eq:zero_charge}. We study the
behavior of
\begin{equation*}
    E_{\xi}(\delta) \equiv \delta\int_{a-\xi}^{a}\int_{-\infty}^{\infty} 
        |\nabla V|^2 \dd y \dd x,
\end{equation*}
where $0 < \xi < a$ is a small parameter. The quantity $E_{\xi}(\delta)$
is proportional to the time-averaged electrical power dissipated in the
strip $R_{\xi} \equiv \{(x,y) \in \mathbb{R}^2 : a-\xi < x < a\}$,
illustrated by the darkened strip in Figure~\ref{fig:slab};
$E_{\xi}(\delta)$ is also a lower bound on the quantity defined in
\eqref{eq:intro_pd}. In particular, we derive conditions on $\rho$ that
determine whether or not $\limsup_{\delta \rightarrow 0^+}
E_{\xi}(\delta) = \infty$ (\emph{weak} CALR), $\lim_{\delta \rightarrow
0^+}E_{\xi}(\delta) = \infty$ (\emph{strong} CALR), or $E_{\xi}(\delta)
< C$ for a constant $C > 0$ as $\delta \rightarrow 0^+$ (no CALR).

In order to do this, we begin by taking the Fourier transform of
\eqref{eq:V_pde} in the $y$\nobreakdash-variable and calculating
$E_{\xi}(\delta)$ explicitly in terms of $\rhohat(x,k)$ (the Fourier
transform of $\rho$ in the $y$\nobreakdash-variable). We then derive
upper and lower bounds on $E_{\xi}(\delta)$ to obtain our results. The
result for unbounded energy is contained in
Corollary~\ref{cor:estimate}. Essentially, if there is a $d_* \in
[d_0,d_1]$ such that
\[
    \ds\limsup_{k \rightarrow \infty}\left| \ee^{d_*k}\int_{d_0}^{d_1}
        \rhohat(x,k)\ee^{-kx}\dd x\right| > 0
\]
and $a < d_* < \tau(\beta)a$, where
\[
	\tau(\beta) = 
	\begin{cases}
		\dfrac{\beta+2}{\beta+1} &\text{for } 0 < \beta < 1, \myspace
		\dfrac{3}{2} &\text{for } \beta \ge 1,
	\end{cases}
\] 
then $\limsup_{\delta \rightarrow 0^+} E_{\xi}(\delta) = \infty$. As far
as we are aware, there are two novelties to our result. First, the
blow-up in energy occurs only if $\rho$ is within a critical distance of
the slab that depends non-trivially on $\beta$. Second, unlike in
Theorem 5.3 of the work by Ammari et al.\ \cite{Ammari:2013:STN} and
Theorem 4.1 in the subsequent work of Ammari et al.\
\cite{Ammari:2013:STNII}, we do not assume that the support of $\rho$ is
completely contained within the critical distance. In fact, there are
examples of charge density distributions $\rho$ that cause a blow-up in
energy if only part of the support of $\rho$ is within the critical
distance --- see
Sections~\ref{subsubsec:rectangle}~and~\ref{subsubsec:circle}. (It seems
the results of Ammari et al.\ \cite{Ammari:2013:STN, Ammari:2013:STNII}
would hold even if only part of the support of $\rho$ is within the
critical distance to the lens --- see the Introduction in their later
work \cite{Ammari:2013:STNII}.) In Theorem~\ref{thm:bounded} we show
that $\lim_{\delta \rightarrow 0^+} E_{\xi}(\delta) = 0$ if $\rho$ is
supported outside the critical distance.  

The remainder of this paper is organized as follows. In
Section~\ref{sec:derivation_V} we derive an expression for the
potential. In Section~\ref{sec:derivation_pd} we compute the power
dissipation $E_{\xi}(\delta)$. In
Section~\ref{sec:preliminary_estimates} we obtain some lower bounds that
are used to prove our result about the blow-up of $E_{\xi}(\delta)$ as
$\delta \rightarrow 0^+$. We then analytically and numerically
illustrate our results for two charge density distributions. In
Section~\ref{sec:bounded_pd} we prove that $E_{\xi}(\delta)$ remains
bounded (and, in fact, goes to 0) as $\delta \rightarrow 0^+$ if $\rho$
is farther than the critical distance from the slab. Finally, in
Section~\ref{sec:bounded_potential} we show that the potential remains
bounded far enough away from the slab in the limit as $\delta\rightarrow
0^+$ regardless of the position of the source.    


\section{Derivation of the Potential}\label{sec:derivation_V}

The potential $V \in
L^2_{\mathrm{loc}}(\mathbb{R}^2)$ solves the following problem in 
the quasistatic regime:
\begin{equation}\label{eq:V_pde_full}
    \left\{
        \begin{aligned}
            &- \nabla \cdot [\varepsilon(x,y)\nabla V(x,y)] = \rho(x,y) 
                \quad \text{in} \ \mathbb{R}^2, \\
            &V(x,y), \, \varepsilon\dfrac{\partial V}{\partial x}(x,y) 
                \quad \text{continuous across} \ x = 0, a \ 
                \text{for almost every} \ y \in \mathbb{R}, \\
            &\dfrac{\partial V}{\partial x}(x,y) \rightarrow 0 
                \quad \text{as} \ |x| \rightarrow \infty 
                \ \text{for almost every} \ y \in \mathbb{R}, \\
            &V(x,\cdot) \in H^1(\mathbb{R}) \quad 
                \text{for almost every} \ x \in \mathbb{R}, \\
            &\dfrac{\partial V}{\partial x}(x,\cdot) \in L^2(\mathbb{R}) 
                \quad\text{for almost every} \ x \in \mathbb{R},         
        \end{aligned}
    \right.
\end{equation}
where $\varepsilon$ is given in \eqref{eq:dc}. In this section, we will
take the Fourier transform with respect to the $y$\nobreakdash -variable
of the problem \eqref{eq:V_pde_full}. Since $V \in
L^2_{\mathrm{loc}}(\mathbb{R}^2)$, the PDE \eqref{eq:V_pde_full} can be
understood in a distributional sense (since $L^2_{\mathrm{loc}}$
functions are distributions \cite{Friedlander:1998:ITD}). The continuity
conditions in \eqref{eq:V_pde_full} ensure continuity of the potential
and the normal component of the electric displacement field $\mathbf{D}
= -\varepsilon \nabla V$ across the left and right edges of the slab.
These continuity conditions are typical in quasistatic problems --- see,
e.g., Section~4.4.2 in the book by Griffiths \cite{Griffiths:1999:ITE}
and the work by Milton et al.\ \cite{Milton:2005:PSQ}. The condition at
infinity in \eqref{eq:V_pde_full} ensures that the $x$\nobreakdash
-component of the electric field, namely $-\partial V/\partial x$,
vanishes as $x \rightarrow \pm\infty$. It turns
out that this condition is sufficient for our purposes (for the problem
stated in \eqref{eq:V_pde_full} one can show that the $y$\nobreakdash
-component of the electric field, namely $-\partial V/\partial y$, goes
to $0$ as $|x| \rightarrow \infty$ as well). We only consider $|x|
\rightarrow \infty$, rather than $x^2 + y^2 \rightarrow \infty$, since
the slab extends infinitely in the $y$\nobreakdash-direction. The last
two requirements are regularity results that we impose to ensure that we
can perform the computations in this section. 

In the remainder of this section, we sketch a proof of the following
theorem; a complete proof can be found in work by one of the authors of 
this paper \cite{Thaler:2014:BVI}.  
\begin{theorem}\label{thm:existence_V}
    There exists a nonempty class of potentials 
    \begin{equation}\label{eq:class}
        \mathcal{V} \equiv \{V \in L^2_{\mathrm{loc}}(\mathbb{R}^2) : 
            V \ \text{satisfies} \ \eqref{eq:V_pde_full}\}.
    \end{equation}
\end{theorem}

We recall the following definitions:
\begin{equation}\label{eq:Rs}
    \left\{\begin{aligned}
        \mathcal{C}&\equiv\{(x,y) \in \mathbb{R}^2 : x < 0\}; \\ 
        \mathring{\mathcal{S}} &\equiv\{(x,y) \in \mathbb{R}^2 : 
            0 < x < a\}; \\ 
        \mathcal{M}&\equiv\{(x,y) \in \mathbb{R}^2 : a < x\}.
    \end{aligned}\right.
\end{equation}
We then define
\begin{equation}\label{eq:def_Vs}
    \left\{\begin{aligned}
        V_c(x,y) &\equiv \chi_{\mathcal{C}}(x,y) V(x,y), \\
        V_s(x,y) &\equiv \chi_{\mathring{\mathcal{S}}}(x,y) V(x,y), \\
        V_m(x,y) &\equiv \chi_{\mathcal{M}}(x,y) V(x,y),
    \end{aligned}\right.
\end{equation}
where 
\begin{equation}\label{eq:def_chi}
    \chi_{U}(x,y) = 
    \begin{cases} 
        1 & \text{if } (x,y) \in U, \\ 
        0 & \text{if } (x,y) \not\in U, 
    \end{cases}
\end{equation}
is the characteristic function of the set $U \subset \mathbb{R}^2$.
Finally, we use the convention that the Fourier transform of a function
$f(x,y)$ with respect to the variable $y$ is defined by
\begin{equation}\label{eq:FT}
    \widehat{f}(x,k) \equiv \int_{-\infty}^{\infty} f(x,y) \ee^{-\ii ky} 
        \dd y.
\end{equation}
If $f$ is a distribution, it is well known \cite{Friedlander:1998:ITD} 
that
\begin{equation}\label{eq:FT_properties}
    \widehat{\frac{\partial f}{\partial x}}(x,k) 
        = \frac{\partial \widehat{f}}{\partial x}(x,k)
    \eqand
    \widehat{\frac{\partial f}{\partial y}}(x,k) =\ii k\widehat{f}(x,k).
\end{equation}

We apply the Fourier transform with respect to $y$ in \eqref{eq:V_pde}
and by straightforward calculations find that the
general form of the Fourier transform of $V_c$ is
\begin{equation}\label{eq:Vc_hat}
    \Vhat_c(x,k) = A_k\ee^{|k|x}
\end{equation}
for arbitrary constants $A_k$.

The continuity conditions at the left boundary of the central slab,
i.e., at $x=0$, together with some algebraic manipulations lead us to
the general form of the Fourier transform of $V_s$, namely
\begin{equation}\label{eq:Vs_hat}
    \Vhat_s(x,k) = \dfrac{A_k}{2\chi_c}\left[(\chi_c+1)\ee^{|k|x} + 
        (\chi_c-1)\ee^{-|k|x}\right],
\end{equation}
where
\begin{equation}\label{eq:chic}
    \chi_c \equiv \varepsilon_s/\varepsilon_c.
\end{equation}

Next we will show the details of the derivation for the solution in the
third layer, $\mathcal{M}$. From \eqref{eq:V_pde_full} we note that in
the set $\mathcal{M}$ the potential satisfies
\begin{equation*}
    \begin{cases} 
        \Delta V_m(x,y) = -\rho(x,y) & \text{for } x > a, \\
        \ds\lim_{x\rightarrow a^+}V_m(x,y) 
            = \ds\lim_{x\rightarrow a^-}V_s(x,y)
            &\text{for almost every} \ y \in \mathbb{R}, \\
        \ds\lim_{x\rightarrow a^+}\varepsilon_m
            \dfrac{\partial V_m}{\partial x}(x,y) 
            = \ds\lim_{x\rightarrow a^-}
            \varepsilon_s\dfrac{\partial V_s}{\partial x}(x,y) 
            &\text{for almost every} \ y \in \mathbb{R}, \\
        \ds\lim_{x\rightarrow \infty}
            \dfrac{\partial V_m}{\partial x}(x,y) = 0 
            &\text{for almost every} \ y \in \mathbb{R}.
    \end{cases} 
\end{equation*}
After taking the Fourier transform with respect to $y$ we find that
$\Vhat_m(x,k)$ satisfies
\begin{equation}\label{eq:Vm_hat_pde}
    \begin{cases} 
        \dfrac{\partial^2\Vhat_m}{\partial x^2}(x,k) - k^2 \Vhat(x,k) 
            = -\rhohat(x,k) & \text{for} \ x > a, \\
        \ds\lim_{x\rightarrow a^+}\Vhat_m(x,k) 
            = \ds\lim_{x\rightarrow a^-}\Vhat_s(x,k) &\text{for all} \ k 
            \in \mathbb{R}, \\
        \ds\lim_{x\rightarrow a^+}\varepsilon_m
            \dfrac{\partial \Vhat_m}{\partial x}(x,k) 
            = \ds\lim_{x\rightarrow a^-}\varepsilon_s
            \dfrac{\partial \Vhat_s}{\partial x}(x,k)
            &\text{for all} \ k \in \mathbb{R}, \\
        \ds\lim_{x\rightarrow \infty}
            \dfrac{\partial \Vhat_m}{\partial x}(x,k) 
            = 0 &\text{for all} \ k \in \mathbb{R}.
    \end{cases} 
\end{equation}

We make the change of variables $z = x-a$ so that
\eqref{eq:Vm_hat_pde} becomes
\begin{equation}\label{eq:Vm_hat_cov}
    \begin{cases} 
        \dfrac{\partial^2\underline{\Vhat_m}}{\partial z^2}(z,k) - 
            k^2 \underline{\Vhat_m}(z,k) = -\underline{\rhohat}(z,k) & 
            \text{for} \ z > 0, \\ 
        \ds\lim_{z \rightarrow 0^+} \underline{\Vhat_m}(z,k) 
            = \ds\lim_{z \rightarrow 0^-} \underline{\Vhat_s}(z,k) 
            = A_k\psi_k^+ 
            &\text{for all} \ k \in \mathbb{R}, \\
        \ds\lim_{z\rightarrow 0^+} \dfrac{\partial 
            \underline{\Vhat_m}}{\partial z}(z,k) 
            = \displaystyle\lim_{z\rightarrow 0^-}\chi_m 
            \dfrac{\partial \underline{\Vhat_s}}{\partial z}(z,k) 
            = A_k\psi_k^- 
            &\text{for all} \ k \in \mathbb{R}, 
    \end{cases} 
\end{equation}
where $\widehat{\rho}(x,k) = \underline{\widehat{\rho}}(x-a,k)$;
$\widehat{V_j}(x,k) = \underline{\widehat{V}_j}(x-a,k)$ for $j = m$,
$s$; 
\begin{align}
    \psi_k^+ &= \dfrac{1}{2\chi_c}
        \left[\left(\chi_c+1\right)\ee^{|k|a}
        + \left(\chi_c-1\right)\ee^{-|k|a}\right]; \label{eq:psiplus} \\
    \psi_k^-  &= \dfrac{|k|\chi_m}{2\chi_c}
        \left[\left(\chi_c+1\right)\ee^{|k|a}
        - \left(\chi_c-1\right)\ee^{-|k|a}\right]; \label{eq:psiminus}\\
    \chi_m &= \varepsilon_s/\varepsilon_m. \nonumber 
\end{align}
(We have eliminated the condition at infinity for now --- we will return
to it later.)  

The Laplace transform of $\underline{\Vhat_m}(z,k)$ is defined by  
\begin{equation}\label{eq:def_Laplace}
    u(s,k) \equiv \int_{0}^{\infty} \underline{\Vhat_m}(z,k) \ee^{-sz} 
        \dd z;
\end{equation}
see, e.g., the book by Schiff \cite{Schiff:1999:TLT}.
We need to solve the ODE in \eqref{eq:Vm_hat_cov} for the cases $k = 0$
and $k \ne 0$ separately.  


\vspace{0.2cm}
\noindent{\bf Case 1:} $k = 0$

\vspace{0.2cm}
Here the Laplace-transformed version of \eqref{eq:Vm_hat_cov} is 
\begin{equation*}
    s^2 u(s,0) - sA_0\psi_0^+ - A_0\psi_0^- 
        = -\mathcal{L}
        \left\{\underline{\widehat{\rho}}(z,0)\right\}(s,0),
\end{equation*}
where $\mathcal{L}\{g\}$ denotes the Laplace transform of the function
$g$ --- see \eqref{eq:def_Laplace}. Thus 
\begin{equation*}
    u(s,0) = \dfrac{A_0}{s} 
        - \left[\mathcal{L}
        \left\{\widehat{\rho}(z,0)\right\}(s,0)\right]\cdot
        \dfrac{1}{s^2}
\end{equation*}
where we have used \eqref{eq:psiplus} and \eqref{eq:psiminus} to
simplify the expression for $u(s,0)$. Since $\underline{\Vhat_m} = 0$
for $z < 0$ (see \eqref{eq:Rs}--\eqref{eq:def_chi}), we can use the
convolution theorem for Laplace transforms \cite{Schiff:1999:TLT}
to find
\begin{equation*}
    \underline{\Vhat_m}(z,0) = A_0 
        - \int_{0}^{z} (z-z')\underline{\widehat{\rho}}(z',0)\dd z' 
        \Rightarrow\Vhat_m(x,0) = A_0 - \ds\int_{0}^{x-a} 
        (x-a-z')\underline{\rhohat}(z',0)\dd z'.
\end{equation*}
Next we make the change of variables $z' = x'-a$ in the above integral
to find 
\begin{equation*}
    \Vhat_m(x,0) = A_0 
        - \ds\int_{a}^{x} (x-x')\underline{\widehat{\rho}}(x'-a,0)\dd x' 
        = A_0 + \ds\int_{a}^{x} (x'-x)\widehat{\rho}(x',0)\dd x'.
\end{equation*}

We now impose the condition as $x\rightarrow \infty$; see
\eqref{eq:Vm_hat_pde}. We need to require
\[
    \lim_{x\rightarrow \infty}\dfrac{\partial \Vhat_m}{\partial x}(x,0) 
    = \ds\lim_{x\rightarrow \infty}\left\{\dfrac{\partial}{\partial x}
    \left[A_0 + \int_{a}^{x}(s-x)\widehat{\rho}(s,0) \dd s\right]
    \right\} = 0.
\]
By the Leibniz Rule \cite{Kaplan:1984:AC, Thaler:2014:BVI},
this is equivalent to the requirement
\begin{equation*}
    \ds\lim_{x\rightarrow \infty} \left[-\ds\int_{a}^{x}
        \widehat{\rho}(s,0) \dd s \right] = 0.
\end{equation*}
For $x > d_1$, by \eqref{eq:zero_charge} we have 
\begin{equation*}
    \int_{a}^{x}\widehat{\rho}(s,0) \dd s 
        = \int_{d_0}^{d_1} \widehat{\rho}(s,0)\dd s 
        = \int_{d_0}^{d_1}\int_{-\infty}^{\infty} 
            \rho(s,y) \dd y\dd s
        = \int_{d_0}^{d_1}\int_{h_0}^{h_1} 
            \rho(s,y) \dd y\dd s
        = 0.
\end{equation*}
Thus the condition at infinity is automatically satisfied for any choice
of $A_0$. (Throughout this section, we have assumed that $\rhohat(x,k)$
is continuous at $k = 0$ --- in fact, in Lemma~\ref{lem:Ik_properties}
we will see that $\rhohat(x,k)$ is infinitely differentiable on
$\mathbb{R}$ as a function of $k$ for almost all $x \in \mathbb{R}$.)  

\vspace{0.2cm}
\noindent{\bf Case 2:} $k \ne 0$
\vspace{0.2cm}

Here the Laplace-transformed version of \eqref{eq:Vm_hat_cov} is 
\[
    s^2u(s,k) - sA_k\psi_k^+ - A_k\psi_k^- - k^2u(s,k) 
        = -\mathcal{L}
            \left\{\underline{\widehat{\rho}}(z,k)\right\}(s,k).
\]
Therefore
\begin{equation*}
    u(s,k) = A_k\psi_k^+\dfrac{s}{s^2-k^2}+A_k\psi_k^-\dfrac{1}{s^2-k^2} 
        -\dfrac{\mathcal{L}\left\{\underline{\widehat{\rho}}(z,k)
        \right\}(s,k)}{s^2-k^2}.
\end{equation*}
Recalling that $\underline{\Vhat_m}(z,k) = 0$ for $z < 0$ (see
\eqref{eq:Rs}--\eqref{eq:def_chi}), by the convolution theorem for
Laplace transforms we have
\begin{equation*}
    \underline{\Vhat_m}(z,k) = A_k\psi_k^+\cosh\left(|k|z\right)
        + A_k\psi_k^-\dfrac{\sinh\left(|k|z\right)}{|k|}
        - \int_{0}^{z}\dfrac{\sinh\left[|k|(z-z')\right]}{|k|}
        \underline{\rhohat}(z',k) \dd z'.
\end{equation*}
This is equivalent to
\begin{align*}
    \Vhat_m(x,k) = &A_k\psi_k^+\cosh\left[|k|(x-a)\right] 
        + A_k\psi_k^-\dfrac{\sinh\left[|k|(x-a)\right]}{|k|}\\*
        & - \ds\int_{0}^{x-a}
        \dfrac{\sinh\left[|k|(x-a-z')\right]}{|k|}
        \underline{\rhohat}(z',k) \dd z'.
\end{align*}
We make the change of variables $z' = x'-a$ in the above integral to
find
\begin{equation}\label{eq:Vm_hat}
    \begin{aligned}
        \Vhat_m(x,k) = &A_k\psi_k^+\cosh\left[|k|(x-a)\right] 
            + \dfrac{A_k\psi_k^-}{|k|}\sinh\left[|k|(x-a)\right] \\*
            &+ \dfrac{1}{|k|}\ds\int_{a}^{x}\sinh\left[|k|(x'-x)\right]
            \widehat{\rho}(x',k) \dd x',
    \end{aligned}
\end{equation}
where we have used the fact that $\underline{\rhohat}(x-a,k) =
\rhohat(x,k)$.

We now impose the limit condition at infinity --- see
\eqref{eq:Vm_hat_pde}. We use the Leibniz Rule to find
\begin{align*}
    &\begin{aligned}
    \lim_{x \rightarrow \infty} 
        \dfrac{\partial \Vhat_m}{\partial x}(x,k) \,
    = &\ds\lim_{x \rightarrow \infty}\Bigg( A_k\left\{|k|\psi_k^+
        \sinh\left[|k|(x-a)\right]+\psi_k^-\cosh\left[|k|(x-a)\right]   
        \right\}\\*
      &-\int_{a}^{x} \widehat{\rho}(x',k)\cosh\left[|k|(x'-x)\right] 
        \dd x' \Bigg)
    \end{aligned}\\
    &\begin{aligned}
    \qquad = &\ds\lim_{x\rightarrow \infty}\Bigg\{|k|\ee^{|k|x}
        \left[\dfrac{A_k\psi_k^+\ee^{-|k|a}}{2}
        +\dfrac{A_k\psi_k^-\ee^{-|k|a}}{2|k|}-\dfrac{1}{2|k|}
        \int_{d_0}^{d_1}\widehat{\rho}(s,k)\ee^{-|k|s}\dd s\right] \\*
        &+\underbrace{|k|\ee^{-|k|x}
        \left[-\dfrac{A_k\psi_k^+\ee^{|k|a}}{2}+\dfrac{A_k\psi_k^-
        \ee^{|k|a}}{2|k|}-\dfrac{1}{2|k|}\displaystyle\int_{d_0}^{d_1}
        \widehat{\rho}(s,k)\ee^{|k|s}\dd s\right]}_{\rightarrow \, 0 
        \ \text{as} \ x \, \rightarrow \, \infty}\Bigg\}
    \end{aligned}\\
    &\begin{aligned}
    \qquad = \ds\lim_{x\rightarrow \infty}\left\{|k|\ee^{|k|x}
        \left[\dfrac{A_k\ee^{-|k|a}}{2|k|}
        \left(|k|\psi_k^+ +\psi_k^-\right)-\dfrac{1}{2|k|}
        \int_{d_0}^{d_1}\widehat{\rho}(s,k)\ee^{-|k|s}\dd s\right]
        \right\}.
    \end{aligned}
\end{align*}
This limit will be $0$ if and only if we choose
\begin{equation}\label{eq:Ak}
    A_k\equiv\dfrac{I_k}{\ee^{-|k|a}\left(|k|\psi_k^+ +\psi_k^-\right)},
\end{equation}
where
\begin{equation}\label{eq:Ik}
    I_k \equiv\int_{d_0}^{d_1}\widehat{\rho}(s,k)\ee^{-|k|s}\dd s.
\end{equation}

By \eqref{eq:dc} and \eqref{eq:chic} we have
\[
    \dfrac{\chi_c-1}{\chi_c+1} = \dfrac{2\ii+\delta-\mu}{\delta+\mu},
\]
so by \eqref{eq:Vs_hat} the potential in the set
$\mathring{\mathcal{S}}$ is
\begin{equation}\label{eq:Vs_hat_simplified}
    \Vhat_s(x,k) = 
    \begin{cases} 
        A_0 & \text{if} \ k = 0, \\
        \dfrac{I_k}{|k|\g}\left[\ee^{|k|x} + \left(\dfrac{2\ii-
        \lambda\delta^{\beta}}{2\delta+\lambda\delta^{\beta}}\right)
        \ee^{-|k|x}\right] & \text{if} \ k \ne 0, 
    \end{cases}
\end{equation}
where
\begin{equation}\label{eq:g}
    \g \equiv \frac{2\chi_c\ee^{-|k|a}\left(\psi_k^+ +
        \frac{1}{|k|}\psi_k^-\right)}{\chi_c+1} =
        \ii\delta\left[1-\dfrac{(\delta+2\ii)
        (2\ii-\lambda\delta^{\beta})}
        {\delta(2\delta+\lambda\delta^{\beta})}\ee^{-2|k|a}\right]
\end{equation}
and $A_0$ is an arbitrary complex constant. In the next section we will
see that the power dissipation is independent of $A_0$. Finally, it can
be shown that \cite{Thaler:2014:BVI}
\begin{equation}\label{eq:lower_bound_g}
    |g|^2 \ge 8\ee^{-4|k|a}\left|\frac{\chi_c}{\chi_c+1}\right|^2
        = 8\left[\frac{1+\delta^2}{(\delta+\mu)^2}\right]\ee^{-4|k|a}
        > 0
\end{equation}
for all $k \in \mathbb{R}$ and all $0 < \delta \le \delta_{\mu}$.  

\begin{remark}\label{rem:constant}
    We may add the term $C\delta(k)$, where $\delta$ is the Dirac delta
    distribution and $C$ is a constant, to $\Vhat$. This corresponds to
    adding a constant to the potential $V$. However, adding a constant
    to the potential will not affect any of the results presented in 
    this paper.    
\end{remark}
\begin{remark}\label{rem:potential}
    By construction, \eqref{eq:Vc_hat}, \eqref{eq:Vs_hat}, 
    \eqref{eq:Vm_hat}, \eqref{eq:Ak}, and \eqref{eq:Ik} characterize 
    all solutions of \eqref{eq:V_pde_full} (at least up to the constant 
    discussed in Remark~\ref{rem:constant}).  
\end{remark}


\section{Derivation of the Power Dissipation}\label{sec:derivation_pd}

We begin this section by recording some important properties of $I_k$,
defined in \eqref{eq:Ik}. We omit the proof of the following lemma since
it is is given elsewhere \cite{Thaler:2014:BVI}.  
\begin{lemma}\label{lem:Ik_properties}
    Suppose $\rho \in \Pset$ (where $\Pset$ is defined in
    \eqref{eq:P_def}) and that $I_k$ is defined as in \eqref{eq:Ik}.
    Then 
    \begin{enumerate}
        \item for almost every $s \in [d_0,d_1]$, $\rhohat(s,k)$ is
            infinitely continuously differentiable as a function of $k$
            for all $k \in \mathbb{R}$; 
        \item for each $k \in \mathbb{R}$, 
            \[
                |I_k|^2\le (d_1-d_0)
                    \left\|\rho\right\|^2_{L^2(\mathcal{M})}
                    \ee^{-2|k|d_0};
            \]
        \item if $\rho$ is real valued, then $I_{-k} = \overline{I_k}$;
            this implies that $|I_k|^2$ is an even function of $k$ for
            $k \in \mathbb{R}$;
        \item the function $I_k$ is continuous at $k$ for each $k \in
            \mathbb{R}$;
        \item $\ds\lim_{k\rightarrow 0} I_k = I_0 = 0$;
        \item $\ds\lim_{k\rightarrow 0}|I_k|/|k| = |C_0| <
            \infty$, where 
            \begin{align*}
                C_0 &= \int_{d_0}^{d_1} -s\rhohat(s,0) \dd s
                        + \int_{d_0}^{d_1} 
                        \frac{\partial \rhohat}{\partial k}(s,0) \dd s\\
                    &= -\int_{d_0}^{d_1} 
                        \int_{h_0}^{h_1}s\rho(s,y) \dd y \dd s
                        - \int_{d_0}^{d_1} \int_{h_0}^{h_1} \ii y 
                        \rho(s,y)\dd y\dd s;
            \end{align*}
            moreover, there is a positive constant $C_I$ 
            such that $|I_k|/|k| \le C_I$ for all $k\in \mathbb{R}$.  
    \end{enumerate}
\end{lemma}

For $0 < \xi < a$, the time-averaged electrical power dissipation in the
strip $R_{\xi}$ is defined as
\begin{equation}\label{eq:def_pd}
    E_{\xi}(\delta) \equiv \delta\int_{a-\xi}^{a}
        \int_{-\infty}^{\infty} |\nabla V(x,y)|^2 \dd y \dd x 
        = \delta\int_{a-\xi}^{a}\ds\int_{-\infty}^{\infty}
        \left(\left|\frac{\partial V}{\partial x}\right|^2 + 
        \left|\frac{\partial V}{\partial y}\right|^2 \right)\dd y \dd x,
\end{equation}
where $V(x,y)$ is the (complex) electric potential in the slab
$\mathcal{S}$ due to the charge density $\rho$ and $|z| =
\sqrt{(z')^2+(z'')^2}$ denotes the modulus of the complex number
$z=z'+\ii z''$. Recall that in the quasistatic regime the potential $V$
solves \eqref{eq:V_pde_full} with $\varepsilon$ given by \eqref{eq:dc}.
Since $V \in H^1(\mathring{\mathcal{S}})$, the quantity in
\eqref{eq:def_pd} is well defined and finite \cite{Thaler:2014:BVI}. 

Using the definition in \eqref{eq:def_pd}, we compute the power
dissipation in the strip $R_{\xi}$ (see Figure~\ref{fig:slab}) as
follows. Note that for any function $f:\mathbb{R}^2 \rightarrow
\mathbb{C}$ such that
\[
    \int_{-\infty}^{\infty} |f(x,y)|^2 \dd y < \infty,
\]
we have the Plancherel Theorem, namely
\begin{equation}\label{eq:Plancherel}
    \int_{-\infty}^{\infty} |f(x,y)|^2 \dd y 
        = \dfrac{1}{2\pi}\ds\int_{-\infty}^{\infty} 
            |\widehat{f}(x,k)|^2 \dd k.
\end{equation}
Using \eqref{eq:Plancherel} together with the classical properties of
the Fourier transform in \eqref{eq:FT_properties}, from
\eqref{eq:def_pd} we obtain
\begin{align}
    E_{\xi}(\delta) &= \delta \int_{a-\xi}^{a}
        \left[\int_{-\infty}^{\infty} 
        \left|\dfrac{\partial V_s}{\partial x}(x,y)\right|^2 \dd y 
        + \int_{-\infty}^{\infty} 
        \left|\dfrac{\partial V_s}{\partial y}(x,y)\right|^2 \dd y
        \right] \dd x \nonumber\\[0.2cm]
    &= \dfrac{\delta}{2\pi} \int_{a-\xi}^{a}
        \left[\int_{-\infty}^{\infty} \left|\dfrac{\partial \Vhat_s}
        {\partial x}(x,k)\right|^2 \dd k + \int_{-\infty}^{\infty}|k|^2|
        \Vhat_s(x,k)|^2 \dd k\right] \dd x. \label{eq:Fourier_pd}
\end{align}
Now, \eqref{eq:Vs_hat_simplified}, \eqref{eq:lower_bound_g}, and
Lemma~\ref{lem:Ik_properties} imply that $\Vhat_s$ and
$\partial \Vhat_s/\partial x$ are finite at and near $k =
0$; thus we can omit the point $k = 0$ from the integrals in
\eqref{eq:Fourier_pd} without changing the value of $E_{\xi}(\delta)$.
Inserting \eqref{eq:Vs_hat_simplified} into \eqref{eq:Fourier_pd} gives
(after some straightforward computations)
\begin{align}
    E_{\xi}(\delta) &= \dfrac{2\delta}{2\pi} \int_{a-\xi}^{a}
        \left\{\int_{k\ne0} \dfrac{|I_k|^2}{|\g|^2}\left[\ee^{2|k|x} 
        + \dfrac{\ee^{-2|k|x}\left(\lambda^2\delta^{2\beta} + 4\right)}
        {(2\delta+\lambda\delta^{\beta})^2}\right] \dd k\right\} \dd x 
        \nonumber\\[0.2cm]
    &= \dfrac{\delta}{\pi} \int_{k\ne0} \dfrac{|I_k|^2}{|\g|^2} 
        \left\{\displaystyle\int_{a-\xi}^{a}\left[\ee^{2|k|x} 
        + \dfrac{\ee^{-2|k|x}\left(\lambda^2\delta^{2\beta} + 4\right)}
        {(2\delta+\lambda\delta^{\beta})^2}\right] \dd x\right\} \dd k 
        \nonumber\\[0.2cm]
    &= \dfrac{\delta}{2\pi} \displaystyle\int_{k\ne0} 
        \dfrac{|I_k|^2}{|k||\g|^2}
        \ee^{2|k|a}\left[\left(1-\ee^{-2|k|\xi}\right) 
        + \dfrac{\left(\lambda^2\delta^{2\beta} + 4\right)}
        {(2\delta+\lambda\delta^{\beta})^2}\ee^{-4|k|a}
        \left(\ee^{2|k|\xi}-1\right)\right] \dd k \nonumber\\[0.2cm]
    &= \dfrac{\delta}{\pi} \int_{k>0} \dfrac{|I_k|^2}{k|\gk|^2}\ee^{2ka}
        \left[\left(1-\ee^{-2k\xi}\right)+
        \dfrac{\left(\lambda^2\delta^{2\beta} 
        + 4\right)}{(2\delta+\lambda\delta^{\beta})^2}\ee^{-4ka}
        \left(\ee^{2k\xi}-1\right)\right] \dd k \label{eq:full_pd}
        \\[0.2cm]
    &\ge \widetilde{E}_{\xi}(\delta) \equiv \int_{k\ge\widetilde{k}} 
        F \dd k , \label{eq:lower_bound_pd}
\end{align}
where $\widetilde{k} > 0$ is arbitrary,
\begin{equation}\label{eq:F}
    F \equiv  
        \left(\dfrac{\delta|I_k|^2}{\pi k|\gk|^2}\right) \ee^{2ka} L,
\end{equation}
and
\begin{equation}\label{eq:L}
    L\equiv\left(1-\ee^{-2k\xi}\right) 
        + \dfrac{\left(\lambda^2\delta^{2\beta} + 4\right)}
        {(2\delta+\lambda\delta^{\beta})^2}\ee^{-4ka}
        \left(\ee^{2k\xi}-1\right).
\end{equation}


\section{Lower Bound on Power Dissipation}%
\label{sec:preliminary_estimates}

In this section we derive some asymptotic estimates on the function $F$
defined in \eqref{eq:F}. From \eqref{eq:g} we have
\begin{equation}\label{eq:mod_g}
    |\gk|^2 = \delta^2\left\{\left(1 + \dfrac{4+\lambda\delta^{\beta+1}}
        {2\delta^2+\lambda\delta^{\beta+1}}\ee^{-2ka}\right)^2 
        + \left[\dfrac{2(\delta-\lambda\delta^{\beta})}{2\delta^2
        +\lambda\delta^{\beta+1}}\ee^{-2ka}\right]^2\right\}.
\end{equation}
Upon inspection of \eqref{eq:full_pd} and \eqref{eq:F} we see
(heuristically) that if $|\gk|^2 = O(\delta^2)$ as $\delta \rightarrow
0^+$, we may be able to show that the power dissipation blows up as
$\delta \rightarrow 0^+$. To this end we define 
\begin{equation}\label{eq:k0}
    k_0(\delta)\equiv\dfrac{1}{2a}
        \ln\left[\dfrac{1}{\delta(\delta+\mu)}\right] 
        = \dfrac{1}{2a}\ln
            \left(\dfrac{1}{2\delta^2+\lambda\delta^{\beta+1}}\right).
\end{equation} 
Note that $k_0(\delta) \rightarrow \infty$ as $\delta \rightarrow 0^+$.
From \eqref{eq:lower_bound_pd} and recalling \eqref{eq:g} and
\eqref{eq:F}--\eqref{eq:L} we see that 
\begin{equation}\label{eq:k0_pd}
    E_{\xi}(\delta) \ge \ds\int_{k_0(\delta)}^{\infty} F \dd k
\end{equation}
for all $0 < \delta \le \delta_0(\beta,\lambda)$ where $0 < \delta_0 \le
\delta_{\mu}$ is such that $k_0(\delta) > 0$ for $0 < \delta \le
\delta_0$. (Recall that $\delta_{\mu}(\beta,\lambda)$ is defined so that
$\mu = \delta + \lambda\delta^{\beta} \ge 0$ for all $\delta \le
\delta_{\mu}$.)
\begin{lemma}\label{lem:upper_bound_g}
    Suppose $\beta > 0$, $\lambda$ is feasible (see
    \eqref{eq:feasible_lambda}), and $C_1 > 25$. Then there exists $0 <
    \delta_g(\beta,\lambda,C_1) \le \delta_{\mu}(\beta,\lambda)$ such
    that if $0 < \delta \le \delta_g$ and $k \ge k_0(\delta)$ then
    \begin{equation*}
        |\gk|^2 \le C_1\delta^2.
    \end{equation*}
\end{lemma}
\begin{proof}
    Note that \eqref{eq:mod_g} is equivalent to
    \begin{equation*}
        \left|\gk\right|^2
        = \delta^2\left[1+\frac{2\left(4+\lambda\delta^{\beta+1}\right)}
            {2\delta^2+\lambda\delta^{\beta+1}}\ee^{-2ka} 
            + \frac{16+4\delta^2+\lambda^2\delta^{2\beta}
            \left(4+\delta^2\right)}
            {\left(2\delta^2+\lambda\delta^{\beta+1}\right)^2}
            \ee^{-4ka}\right].
    \end{equation*}
    All three terms in the above equation are positive for all $0 <
    \delta \le \delta_{\mu}$. Also, since $k \ge k_0(\delta)$,
    $\ee^{-2ka} \le \ee^{-2k_0a} = 2\delta^2+\lambda\delta^{\beta+1}$.
    Then for $0 < \delta \le \delta_{\mu}$ we have
    \begin{equation*}
        \left|\gk\right|^2 
            \le \delta^2\left[25+2\lambda\delta^{\beta+1} 
            + 4\delta^2+\lambda^2\delta^{2\beta}
            \left(4+\delta^2\right)\right].
    \end{equation*}
    We then choose $\delta_g(\beta,\lambda,C_1)\le
    \delta_{\mu}(\beta,\lambda)$ small enough to ensure that the term in
    brackets is less than or equal to $C_1$ for all $0 < \delta \le
    \delta_g$.
\end{proof}
\begin{lemma}\label{lem:lower_bound_L}
    Suppose $\beta > 0$, $\lambda$ is feasible, $0 < \xi < a$, and let
    $0 < C_L < 1$ be a constant. Then there exists $0 <
    \delta_L(\beta,\lambda,\frac{\xi}{a},C_L) \le
    \delta_{\mu}(\beta,\lambda)$ such that if $0 < \delta \le \delta_L$
    and $k \ge k_0(\delta)$, then $L \ge C_L$.
\end{lemma}
\begin{proof}
    From \eqref{eq:L} we have
    \begin{align*}
        L
        &= \left(1-\ee^{-2k\xi}\right)+\frac{\lambda^2\delta^{2\beta}+4}
            {\left(2\delta+\lambda\delta^{\beta}\right)^2}\ee^{-4ka}
            \left(\ee^{2k\xi}-1\right) \\
        &\ge 1-\ee^{-2k\xi} \ge 1-\ee^{-2k_0\xi} 
            = 1-\left(2\delta^2+\lambda\delta^{\beta+1}\right)
            ^{\frac{\xi}{a}} \ge C_L
    \end{align*}
    for $0 < \delta \le \delta_L(\beta,\lambda,\frac{\xi}{a},C_L)$,
    where $0 < \delta_L \le \delta_{\mu}$ is such that
    $\left(2\delta^2+\lambda\delta^{\beta+1}\right)^{\frac{\xi}{a}} \le
    1-C_L$ for $0 < \delta \le \delta_L$.
\end{proof}
For $0 < \delta \le \min\{\delta_0,\delta_g,\delta_L\}$ we apply the
bounds from Lemmas~\ref{lem:upper_bound_g}~and~\ref{lem:lower_bound_L}
to \eqref{eq:k0_pd} and, recalling \eqref{eq:F}--\eqref{eq:L} and
\eqref{eq:mod_g}, find
\begin{equation}\label{eq:lower_bound_pd_g_L}
    E_{\xi}(\delta) \ge \dfrac{C_L}{\pi C_1\delta} 
        \int_{k_0(\delta)}^{\infty} \dfrac{|I_k|^2}{k} \ee^{2ka} \dd k.
\end{equation}
Our goal is to show that $E_{\xi}(\delta)$ tends to infinity as a
sequence $\delta_j$ tends to 0.  

Since $|I_k|^2$ is a continuous function of $k$ (by
Lemma~\ref{lem:Ik_properties}), from \eqref{eq:lower_bound_pd_g_L} and
the Mean Value Theorem for Integrals we have, for $0 < \delta \le
\min\{\delta_0,\delta_g,\delta_L\}$, that
\begin{align}
     E_{\xi}(\delta) &\ge \dfrac{C_L}{\pi C_1\delta} 
        \int_{k_0(\delta)}^{k_0(\delta)+\frac{1}
        {\ln\left(\frac{\ee}{\delta}\right)}} \dfrac{|I_k|^2}{k} 
        \ee^{2ka} \dd k \nonumber \\
     &\ge \left(\dfrac{C_L}{\pi C_1}\right)
        \left(\dfrac{\ee^{2k_0(\delta)a}}
        {\delta\left[k_0(\delta)+1\right]}\right) 
        \int_{k_0(\delta)}^{k_0(\delta)+
        \frac{1}{\ln\left(\frac{\ee}{\delta}\right)}}|I_k|^2 \dd k 
        \nonumber \\
     &= \left(\dfrac{C_L}{\pi C_1}\right)
        \left[\dfrac{\ee^{2k_0(\delta)a}}
        {\delta\ln\left(\frac{\ee}{\delta}\right)
        \left[k_0(\delta)+1\right]}\right] 
        |I_{k_0(\delta)+t(\delta)}|^2 \label{eq:G}
\end{align}
for some $0 \le t(\delta) \le
1/\ln\left(\ee/\delta\right) \le 1$. Note that
$t(\delta) \rightarrow 0$ as $\delta \rightarrow 0^+$. So now we must
show the lower bound \eqref{eq:G} tends to infinity as a sequence
$\delta_j$ tends to 0.
\begin{theorem}\label{thm:estimate}
    Let $\rho \in \Pset$, $\beta > 0$, and $\lambda$ be feasible. Assume
    there exist constants $d_* \in [d_0,d_1]$ and $ \Lambda \in
    (0,\infty]$ such that $\limsup_{k\rightarrow
    \infty}|I_k\ee^{kd_*}|=\Lambda$. Then there exists a sequence
    $\{\delta_j\}_{j=1}^{\infty}$ with $\delta_j \rightarrow 0$ as $j
    \rightarrow \infty$ and there exist positive constants $C' =
    \frac{C_L\ee^{-2d_*}}{2\pi C_1}$, $C_2 =
    \frac{C'a\Lambda^2\lambda^{(d_*-a)/a}}{2}$, $C_3 = \ln\lambda$, and
    $C_4 = \frac{C'a\Lambda^2}{4}$ such that 
    \begin{equation}\label{eq:estimate}
        E_{\xi}(\delta_j) \ge
        \begin{cases}
            \dfrac{C_2\delta_j^{(\beta+1)
                \left(\frac{d_*-a}{a}\right)-1}}
                {\left(\ln\delta_j-1\right)
                \left[C_3 + (\beta+1)\ln\delta_j\right]} 
                &\text{for} \ 0 < \beta < 1, \myspace
        	\dfrac{C_4\delta_j^{2\left(\frac{d_*-a}{a}\right)-1}}
	            {\left(\ln\delta_j - 1\right)\ln\delta_j} &\text{for} \ 
		        \beta \ge 1.
        \end{cases}
    \end{equation}
    (The constants $C_2$ and $C_3$ are well defined since we require
    $\lambda > 0$ if $0 < \beta < 1$ --- see
    \eqref{eq:feasible_lambda}.) Moreover, if $\lim_{k\rightarrow
    \infty}|I_k\ee^{kd_*}| = \Lambda$, then for $\delta$ small enough we
    have
    \begin{equation}\label{eq:continuous_estimate}
        E_{\xi}(\delta) \ge
        \begin{cases}
            \dfrac{C_2\delta^{(\beta+1)\left(\frac{d_*-a}{a}\right)-1}}
                {\left(\ln\delta-1\right)\left[C_3 + (\beta+1)\ln\delta
        	    \right]} &\text{for} \ 0 < \beta < 1, \myspace
		    \dfrac{C_4\delta^{2\left(\frac{d_*-a}{a}\right)-1}}
        	    {\left(\ln\delta-1\right)\ln\delta} &\text{for}\ \beta\ge 1.
        \end{cases}
    \end{equation}
\end{theorem}
\begin{proof}
    If $0 < \delta \le \min\{\delta_0,\delta_g,\delta_L\}$, then
    \eqref{eq:G} holds. Since $0 \le t(\delta) \le 1$ and $k_0(\delta) +
    1 \le 2k_0(\delta)$ for $\delta$ small enough (equivalently
    $k_0(\delta)$ large enough), \eqref{eq:G} implies
    \begin{align}
        E_{\xi}(\delta) &\ge \left(\dfrac{C_L}{2\pi C_1}\right)
            \left[\dfrac{\ee^{2k_0(\delta)a}}
            {\delta\ln\left(\frac{\ee}{\delta}\right)k_0(\delta)}\right]
            \left|[I_{k_0(\delta)+t(\delta)}]
            \ee^{[k_0(\delta)+t(\delta)]d_*}\right|^2
            \ee^{-2[k_0(\delta)+t(\delta)]d_*} \nonumber\\
        &\ge \dfrac{C'\ee^{-2k_0(\delta)(d_*-a)}}
            {\delta\ln\left(\frac{\ee}{\delta}\right)
            k_0(\delta)} \left|I_{k'(\delta)}
            \ee^{k'(\delta)d_*}\right|^2, \label{eq:intermediate_again}
    \end{align}
    where $k'(\delta) \equiv k_0(\delta) + t(\delta)$. 

    Since $\ds\limsup_{k\rightarrow\infty}|I_k\ee^{kd_*}| = \Lambda$
    there exists a sequence $\{k_j\}_{j=1}^{\infty}$ with $k_j
    \rightarrow \infty$ as $j \rightarrow \infty$ and 
    \[
        \ds\lim_{j\rightarrow \infty}|I_{k_j}\ee^{k_jd_*}| = \Lambda.
    \]
    We choose a sequence $\{\delta_j\}_{j=1}^{\infty}$ such that
    $\delta_j \rightarrow 0^+$ as $j \rightarrow \infty$ and $k_j =
    k_0(\delta_j)$ (where $k_0(\delta) =
    -\frac{1}{2a}\ln(2\delta^2+\lambda\delta^{\beta+1}) $ is defined in
    \eqref{eq:k0}).    
	
    Since $|I_k\ee^{kd_*}|$ is a continuous function of $k$ and
    $t(\delta_j) \rightarrow 0$ as $j \rightarrow \infty$ (i.e., as
    $\delta_j \rightarrow 0^+$), we have
    \[
        \ds\lim_{j\rightarrow \infty} |I_{k'_j}\ee^{k'_jd_*}| = \Lambda,
    \]
    where $k'_j = k_0(\delta_j) + t(\delta_j) = k_j + t(\delta_j)
    \rightarrow \infty$ as $j \rightarrow \infty$. Thus, for $j$ large
    enough (i.e., $\delta_j$ small enough), $|I_{k'_j}\ee^{k'_jd_*}| \ge
    \frac{\Lambda}{2}$. Hence for large enough $j$ we have
    \begin{equation}\label{eq:last_lb}
        E_{\xi}(\delta_j) \ge \left(\dfrac{C'\Lambda^2}{4}\right)
            \dfrac{\ee^{-2k_0(\delta_j)(d_*-a)}}
            {\delta_j\ln\left(\frac{\ee}{\delta_j}\right)k_0(\delta_j)} 
            = \left(\dfrac{C'\Lambda^2}{4}\right)
            \dfrac{\left(2\delta_j^2+\lambda\delta_j^{\beta+1}\right)
            ^{(d_*-a)/a}}{\delta_j\ln\left(\frac{\ee}{\delta_j}\right)
            k_0(\delta_j)}.
    \end{equation}
    Now \eqref{eq:estimate} is obtained by applying the inequality 
        \begin{equation*}
            2\delta_j^2 + \lambda\delta_j^{\beta+1} \ge 
            \begin{cases}
                \lambda\delta_j^{\beta+1} &\text{for} \ 0 < \beta < 1,\\
                \delta_j^2 &\text{for}\ \beta \ge 1,
            \end{cases}
        \end{equation*}
    which holds for $j$ large enough, to \eqref{eq:last_lb}.
	
    Similarly, if the stronger condition $\lim_{k\rightarrow
    \infty}|I_k\ee^{kd_*}| = \Lambda$ holds, since
    $k'(\delta)\rightarrow \infty$ as $\delta \rightarrow 0^+$ we have
    $|I_{k'(\delta)}\ee^{k'(\delta)d_*}| \ge \frac{\Lambda}{2}$ and 
    \begin{equation}\label{eq:last_lb_continuous}
        E_{\xi}(\delta) \ge \left(\dfrac{C'\Lambda^2}{4}\right)
            \dfrac{\left(2\delta^2+\lambda\delta^{\beta+1}\right)
            ^{(d_*-a)/a}}
            {\delta\ln\left(\frac{\ee}{\delta}\right)k_0(\delta)}
    \end{equation}
    for $\delta$ small enough; this is the continuous analog of
    \eqref{eq:last_lb} and is a direct consequence of
    \eqref{eq:intermediate_again}. Finally,
    \eqref{eq:continuous_estimate} is obtained by inserting the
    inequality
    \begin{equation*}
        2\delta^2 + \lambda\delta^{\beta+1} \ge 
        \begin{cases}
            \lambda\delta^{\beta+1} &\text{for} \ 0 < \beta < 1, \\
            \delta^2 &\text{for} \ \beta \ge 1,
        \end{cases}
    \end{equation*}
    which holds for $\delta$ small enough, into
    \eqref{eq:last_lb_continuous}.
\end{proof}

The next corollary follows immediately.
\begin{corollary}\label{cor:estimate}
    Let $\rho \in \Pset$, $\beta > 0$, and $\lambda$ be feasible. Assume
    there exist constants $d_* \in [d_0,d_1]$ and $ \Lambda \in
    (0,\infty]$ such that
    \begin{enumerate}
        \item[(a)] $\limsup\limits_{k\rightarrow \infty}|I_k\ee^{kd_*}|
            =\Lambda$; or
        \item[(b)] $\ds\lim_{k\rightarrow\infty}|I_k\ee^{kd_*}|
            =\Lambda$.
    \end{enumerate}
    If $d_* < \tau(\beta)a$, where $\tau$ is the continuous function 
    \begin{equation}\label{eq:tau}
        \tau(\beta)\equiv
        \begin{cases} 
            \dfrac{\beta+2}{\beta+1} &\text{if} \ 0 < \beta < 1,\myspace
            \dfrac{3}{2} &\text{if} \ \beta \ge 1,
        \end{cases}
    \end{equation}
    then $\limsup_{\delta \rightarrow 0^+}E_{\xi}(\delta) = \infty$ if
    (a) holds (weak CALR) and $\lim_{\delta \rightarrow
    0^+}E_{\xi}(\delta) = \infty$ if (b) holds (strong CALR).
\end{corollary}
\begin{remark}\label{rem:region}
    According to the previous corollary, the region of influence, i.e.,
    the region in which the charge density $\rho$ should be placed to
    cause the anomalous localized resonance near the inner right edge of
    the slab, is the interval $(a,\tau(\beta)a)$. In particular we can
    take $d_1 < \tau(\beta)a$ to guarantee that $\rho$ is completely
    inside this region (assuming the support of $\rho$ is small enough
    so that $d_0 > a$ as well). This region of influence is the same as
    that found in the cloaking paper by Milton and Nicorovici
    \cite{Milton:2006:CEA} and also in the superlensing paper by Milton,
    Nicorovici, McPhedran, and Podolskiy \cite{Milton:2005:PSQ} in
    the particular case when $\rho$ is a dipole source. Also see
    Bergman's work \cite{Bergman:2014:PIP}. 
\end{remark}


\subsection{Numerical Discussion}\label{subsec:numerical_discussion}

In this section, we study the behavior of two charge density
distributions $\rho$. In particular, we show they satisfy the conditions
of Theorem~\ref{thm:estimate} that lead to weak CALR, i.e., they satisfy
$\limsup_{k\rightarrow \infty}|I_k\ee^{kd_*}|=\Lambda$. We also provide
plots illustrating the blow-up of the dissipated electrical power as
$\delta \rightarrow 0^+$ for these charge density distributions.


\subsubsection{Rectangle}\label{subsubsec:rectangle}

The first charge density distribution we consider has support in a
rectangle centered at $(x_0,y_0)$. The left and right edges of the
rectangle are at $d_0 = x_0-d$ and $d_1 = x_0+d$, respectively, where $d
> 0$. The bottom and top edges are at $h_0 = y_0-h$ and $h_1 = y_0+h$,
respectively, where $h > 0$. These parameters are chosen so $d_0 > a$.
We define the charge density distribution as
\begin{equation*}
    \rho(x,y) = 
    \begin{cases}
        Q &\text{for} \ (x,y) \in [d_0,d_1] \times (y_0,h_1], \\
        -Q &\text{for} \ (x,y) \in [d_0,d_1] \times [h_0,y_0), \\
        0 &\text{otherwise,}
    \end{cases}
\end{equation*}
where $Q \ne 0$. Since $\rho \in L^1(\mathcal{M})\cap L^2(\mathcal{M})$,
we can use calculus and \eqref{eq:FT} and \eqref{eq:Ik} to find
\begin{align*}
    \rhohat(x,k) &= -\frac{4Q}{k}\left[\sin(y_0k) + \ii\cos(y_0k)\right]
        \sin^2\left(\frac{hk}{2}\right) 
    \intertext{and}
    |I_k| &= \frac{4|Q|}{k^2}\sin^2\left(\frac{hk}{2}\right)\ee^{-d_0k}
        \left(1-\ee^{-2dk}\right). 
\end{align*}
If we take $k_j = \frac{(2j-1)\pi}{h}$ for $j = 1$, $2$, $\ldots$ and
$d_* = d_0 + \alpha$ for $\alpha > 0$ we have
\begin{equation*}
    |I_{k_j}\ee^{d_*k_j}| = \frac{4|Q|}{k_j^2}\ee^{\alpha k_j}
        \left(1-\ee^{-2dk_j}\right) \rightarrow \infty \quad 
        \text{as} \ j \rightarrow \infty.
\end{equation*}
This implies $\limsup_{k\rightarrow \infty} |I_k\ee^{d_*k}| = \infty$,
so $\rho$ satisfies the conditions of Theorem~\ref{thm:estimate}. Thus
there is a sequence $\delta_j \rightarrow 0$ as $j \rightarrow \infty$
such that $E_{\xi}(\delta_j) \rightarrow \infty$ as $j \rightarrow
\infty$ if $d_0 + \alpha < \tau(\beta)a$; according to
Theorem~\ref{thm:bounded} in the next section, if $d_0 > \tau(\beta)a$,
then $E_{\xi}(\delta) \rightarrow 0$ as $\delta \rightarrow 0^+$.  

Since $\alpha > 0$ is arbitrary, the limit superior of the power
dissipation blows up as the dissipation in the lens tends to 0 as long
as any part of the charge density distribution $\rho$ is within the
region of influence $(a,\tau(\beta)a)$.

In Figure~\ref{fig:rectangle_blow_up} we plot $E_{\xi}(\delta)$ for the
rectangular charge density $\rho$ studied above for various values of
$\beta$ and $\delta$. The support of $\rho$ is centered at $(6,6)$, and
has width and height 2; thus $d_0 = h_0 = 5$ and $d_1 = h_1 = 7$. We
take $0 < \beta < 1$ and $a = d_1/\tau(\beta) =
d_1[(\beta+1)/(\beta+2)]$, so the support of $\rho$ is completely
inside the region of influence (see \eqref{eq:tau} and the remark
following it). Figure~\ref{fig:rectangle_blow_up}(a) is a plot of the
power dissipation $E_{\xi}(\delta)$ as a function of $\beta$ and
$\delta$. We observe the divergence of $E_{\xi}(\delta)$ as $\delta
\rightarrow 0^+$ for $0 < \beta < 1$; in particular the divergence
appears to be more severe for larger values of $\beta$. In
Figure~\ref{fig:rectangle_blow_up}(b) we fix $\delta = 10^{-16}$ and
plot $E_{\xi}(\delta)$ as a function of $\beta$. Note the strong
dependence of the divergence of $E_{\xi}(\delta)$ on the relative
dissipation parameter $\beta$. Finally, in
Figure~\ref{fig:rectangle_blow_up}(c) we plot $E_{\xi}(\delta)$ as a
function of $\delta$ for $\beta = 0.8$.  
\begin{figure}[!bt]
    \begin{center}
        \includegraphics{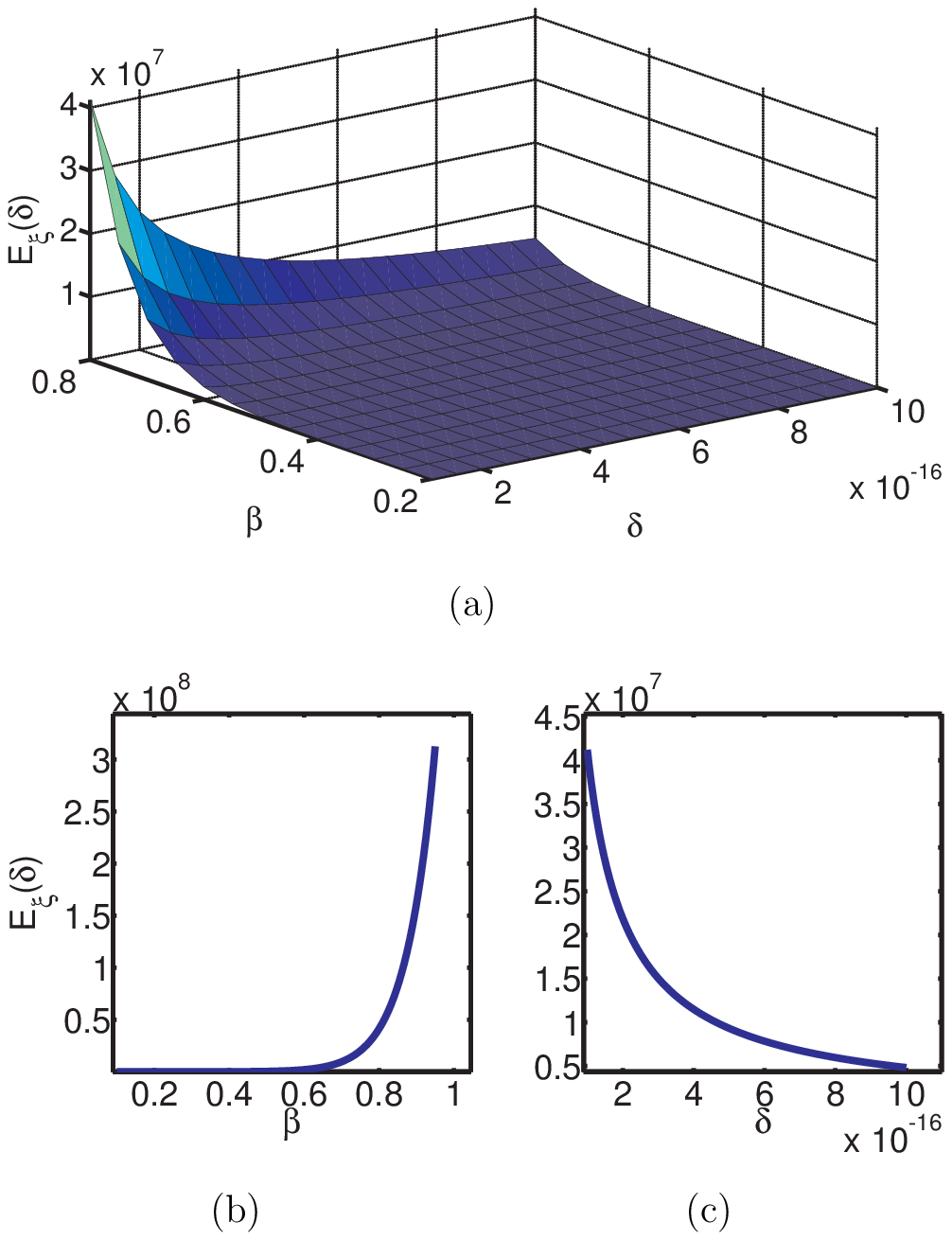}
    \end{center}
    \caption[Rectangular Charge Density Blow-Up]
    {(Rectangular $\rho$) In all of these subfigures we take $a =
    d_1/\tau(\beta)$ so $\rho$ is completely within the region of
    influence. (a) A plot of $E_{\xi}(\delta)$ versus $\beta$ and
    $\delta$ --- the $z$\nobreakdash -axis scale is $10^7$; (b) a plot
    of $E_{\xi}(\delta)$ for $\delta = 10^{-16}$ as a function of
    $\beta$ --- the $y$\nobreakdash -axis scale is $10^8$; (c) a plot of
    $E_{\xi}(\delta)$ for $\beta = 0.8$ as a function of $\delta$ ---
    the $y$\nobreakdash -axis scale is $10^7$.}%
    \label{fig:rectangle_blow_up}
\end{figure}


\subsubsection{Circle}\label{subsubsec:circle}

We now consider a charge density distribution with support in a circle
of radius $R$ centered at $(x_0,y_0)$. In this case we have $d_0 =
x_0-R$, $d_1 = x_0+R$, $h_0(x) = y_0-\sqrt{R^2-(x-x_0)^2}$, and $h_1(x)
= y_0+\sqrt{R^2-(x-x_0)^2}$. Again we choose the parameters so that $d_0
> a$. We define the charge density distribution as
\begin{equation*}
    \rho(x,y) = 
    \begin{cases}
        Q &\text{for} \ d_0 \le x \le d_1, y_0 < y \le h_1(x), \\
        -Q &\text{for} \ d_0 \le x \le d_1, h_0(x) \le y < y_0, \\
        0 &\text{otherwise,}
    \end{cases}
\end{equation*}
where $Q \ne 0$. Again, $\rho \in L^1(\mathcal{M}) \cap
L^2(\mathcal{M})$, so \eqref{eq:FT} and \eqref{eq:Ik} imply
\begin{align*}
    \rhohat(x,k) &= -\frac{4Q}{k}\left[\sin(y_0k) + \ii\cos(y_0k)\right]
        \sin^2\left[\frac{k}{2}\sqrt{R^2-(x-x_0^2)}\right] 
    \intertext{and}
    |I_k| &= \frac{4|Q|}{k}\int_{d_0}^{d_1}\sin^2\left[\frac{k}{2}
        \sqrt{R^2-(x-x_0^2)}\right]\ee^{-kx} \dd x. 
\end{align*}
\emph{Claim}: If $d_* = x_0 + \alpha$ for $\alpha > 0$, then 
$\limsup_{k\rightarrow \infty} |I_k\ee^{d_*k}| = \infty$. \\
\begin{proof}[Proof of Claim]
    Let $\{k_j\}_{j=1}^{\infty}$ be the sequence whose $j^{\mathrm{th}}$
    term is given by
    \[
        k_j = \frac{2}{R}\left(\frac{\pi}{2} + 2\pi j\right).
    \] 
    Then
    \begin{equation}\label{eq:circle_lower_bound_I_k}
        |I_{k_j}| \ge \frac{4|Q|}{k_j}\int_{x_0}^{x_0 + \gamma_j}\sin^2
            \left[\frac{k_j}{2}\sqrt{R^2-(x-x_0^2)}\right]\ee^{-k_jx} 
            \dd x,
    \end{equation} 
    where $\gamma_j = \frac{R}{j}$ for $j = 1, 2, \ldots$.

    For $x\in[x_0,x_0+\gamma_j]$ we have
    \begin{equation}\label{eq:circle_x_bounds}
        \frac{k_j}{2}\sqrt{R^2-\gamma_j^2} \le 
        \frac{k_j}{2}\sqrt{R^2-(x-x_0)^2} \le \frac{k_jR}{2}.
    \end{equation}
    We also have
    \begin{equation*}
        \frac{k_j}{2}\sqrt{R^2-\gamma_j^2} 
            = \left(\frac{\pi}{2} + 2\pi j\right)\sqrt{1-\frac{1}{j^2}} 
            = \frac{\pi}{2} - \zeta_j + 2\pi j,
    \end{equation*}
    where
    \[
        \zeta_j \equiv \frac{\frac{\pi}{2} + 2\pi j}
            {j^2\left(1+\sqrt{1-\frac{1}{j^2}}\right)} 
            = \left(\frac{\pi}{2} + 2\pi j\right)
            \left(1-\sqrt{1-\frac{1}{j^2}}\right).
    \]
    Note $\zeta_j \rightarrow 0^+$ as $j \rightarrow \infty$ so that $0
    < \zeta_j < \pi/2$ for $j$ large enough. In combination with
    \eqref{eq:circle_x_bounds} this implies
    \begin{equation}\label{eq:circle_x_bounds_3}
        2\pi j < \frac{\pi}{2} - \zeta_j + 2\pi j \le \frac{k_j}{2}
            \sqrt{R^2-(x-x_0)^2} \le \frac{k_jR}{2} 
            = \frac{\pi}{2}+2\pi j
    \end{equation}
    for $j$ large enough. Since $\sin\theta$ is monotone increasing for
    $\theta \in (0,\pi/2$), \eqref{eq:circle_lower_bound_I_k}
    and \eqref{eq:circle_x_bounds_3} imply
    \begin{align*}
        |I_{k_j}| &\ge \frac{4|Q|}{k_j}\sin^2\left(\frac{\pi}{2}-\zeta_j
            +2\pi j\right)\int_{x_0}^{x_0+\gamma_j}\ee^{-k_j x} \dd x\\
            &=\frac{4|Q|}{k_j^2}\sin^2\left(\frac{\pi}{2}-\zeta_j\right)
            \ee^{-x_0k_j}\left(1-\ee^{-\gamma_jk_j}\right).
    \end{align*}
    Hence for $j$ large enough we have
    \begin{equation*}
        |I_{k_j}\ee^{d_*k_j}| \ge 
            \frac{4|Q|}{k_j^2}\sin^2\left(\frac{\pi}{2}-\zeta_j\right)
            \ee^{\alpha k_j}\left(1-\ee^{-\gamma_jk_j}\right) \ge 
            |Q|\left(1-\ee^{-4\pi}\right)\frac{\ee^{\alpha k_j}}{k_j^2};
    \end{equation*}
    this expression goes to $\infty$ as $j \rightarrow \infty$.  
    Thus $\limsup_{k\rightarrow \infty}|I_k\ee^{d_*k}| =
    \infty$.
\end{proof}

\begin{figure}[!b]
    \begin{center}
        \includegraphics{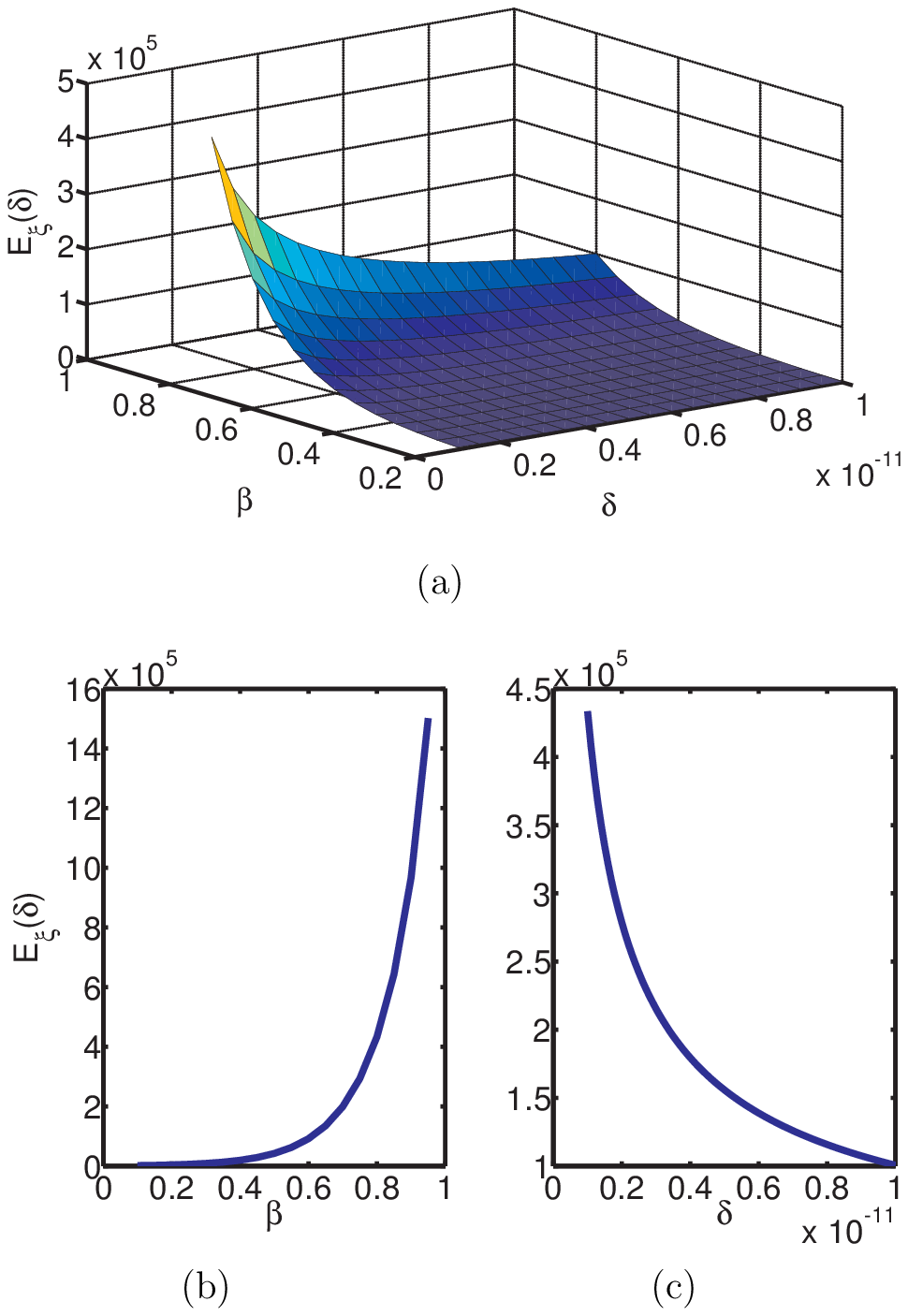}
    \end{center}
    \caption[Circular Charge Density Blow-Up]
    {(Circular $\rho$) In all of these subfigures we take $a =
    d_1/\tau(\beta)$ so $\rho$ is completely within the region of
    influence. (a) A plot of $E_{\xi}(\delta)$ versus $\beta$ and
    $\delta$ --- the $z$\nobreakdash -axis scale is $10^5$; (b) a plot
    of $E_{\xi}(\delta)$ for $\delta = 10^{-12}$ as a function of
    $\beta$ --- the $y$\nobreakdash -axis scale is $10^5$; (c) a plot of
    $E_{\xi}(\delta)$ for $\beta = 0.8$ as a function of $\delta$ ---
    the $y$\nobreakdash-axis scale is $10^5$.}%
    \label{fig:circle_blow_up}
\end{figure}

In Figure~\ref{fig:circle_blow_up} we plot $E_{\xi}(\delta)$ as a
function of $\beta$ and $\delta$ for the circular charge distribution
discussed above. We assume $\rho$ is centered at $(6,6)$ so $d_0 = 5$
and $d_1 = 7$ as in the rectangular case. The only other difference
between Figures~\ref{fig:circle_blow_up}~and~\ref{fig:rectangle_blow_up}
are the values of $\delta$ we used to construct the plots.  

Again we note that $\rho$ need not be completely within the region of
influence for the limit superior of the power dissipation to blow-up as
the dissipation in the lens goes to 0. In particular, according to the
above analysis, $\rho$ only needs to be slightly more than halfway
inside the region of influence for the blow-up to occur. However,
numerical results seem to indicate that the power dissipation due to
this charge density distribution blows up even if $\rho$ is just inside
the region of influence (as is the case for the rectangular charge
density distribution analyzed in Section~\ref{subsubsec:rectangle}).


\section{Upper Bound Power Dissipation}\label{sec:bounded_pd}

In this section, we discuss what happens when $d_0 > \tau(\beta)a \ge
(3/2)a$. Recall that $\rho$ has compact support, so $\supp (\rho)
\subseteq [d_0,d_1]\times[h_0,h_1]$ for some constants $h_0 < h_1$.  
The power dissipation is given \emph{exactly} by
\begin{equation*}
    E_{\xi}(\delta) =   \displaystyle\int_{k>0} F\dd k ;
\end{equation*}
see \eqref{eq:full_pd} and \eqref{eq:F}--\eqref{eq:L}. We will now prove
a series of lemmas that will lead to an upper bound on
$E_{\xi}(\delta)$.    
\begin{lemma}\label{lem:g2bounded}
    Suppose $\beta > 0$ and $\lambda$ is feasible, and let $k_0(\delta)$
    be defined as in \eqref{eq:k0}. Then for every $0 < \delta \le
    \delta_{0}$
    \begin{equation*}
        |g|^2 \ge 
        \begin{cases} 
            9\ee^{-4ka}\dfrac{\delta^2}{\left(2\delta^2
                +\lambda\delta^{\beta+1}\right)^2} 
                &\text{for}\ 0 \le k \le k_0(\delta), \myspace
            \ee^{-ka} \dfrac{\delta^2}
                {\left(2\delta^2+\lambda\delta^{\beta+1}\right)
                ^{\frac{1}{2}}} &\text{for}\ k \ge k_0(\delta).
        \end{cases}
    \end{equation*}
\end{lemma}
\begin{proof}
    From \eqref{eq:mod_g} we have
    \begin{align}
        |\gk|^2 &= \delta^2
            \left\{\left(1 + \dfrac{4+\lambda\delta^{\beta+1}}
            {2\delta^2+\lambda\delta^{\beta+1}}\ee^{-2ka}\right)^2 
            + \left[\dfrac{2(\delta-\lambda\delta^{\beta})}{2\delta^2
            +\lambda\delta^{\beta+1}}\ee^{-2ka}\right]^2\right\} 
            \nonumber\\
        &\ge \delta^2\left(1 + \dfrac{4+\lambda\delta^{\beta+1}}
            {2\delta^2+\lambda\delta^{\beta+1}}\ee^{-2ka}\right)^2. 
            \label{eq:intermediate_g}
    \end{align}
    For $0 < \delta \le \delta_0 \le \delta_{\mu} < 1$ (which implies
    $\mu = \delta + \lambda \delta^{\beta} \ge 0$) we have
    $4+\lambda\delta^{\beta+1} \ge 4-\delta^2 \ge 4-\delta_{\mu}^2 \ge
    3$. Then, from \eqref{eq:intermediate_g}, for fixed $\delta \le
    \delta_{0}$, and for all $k \in \mathbb{R}$ we have
    \begin{equation*}
        |\gk|^2 \ge \delta^2
            \left(\dfrac{3}{2\delta^2 + \lambda\delta^{\beta+1}}
            \ee^{-2ka}\right)^2 = 9\ee^{-4ka}\dfrac{\delta^2}
            {\left(2\delta^2+\lambda\delta^{\beta+1}\right)^2}.
    \end{equation*}
    In particular this bound holds for $0 \le k \le k_0(\delta)$.
    
    To prove the second part of the lemma we note
    \eqref{eq:intermediate_g} implies $|\gk|^2 \ge \delta^2$ when $0 <
    \delta \le \delta_{\mu}$. If $k \ge k_0(\delta)$ holds as well we
    have
    \begin{equation*}
        \ee^{-ka} \dfrac{\delta^2}{\left(2\delta^2
            +\lambda\delta^{\beta+1}\right)^{\frac{1}{2}}} 
            \le \ee^{-k_0(\delta)a} \dfrac{\delta^2}
            {\left(2\delta^2+\lambda\delta^{\beta+1}\right)
            ^{\frac{1}{2}}} 
            = \delta^2 \le |\gk|^2.
    \end{equation*}
\end{proof}
Combining the computations from
Lemmas~\ref{lem:Ik_properties}~and~\ref{lem:g2bounded} we find, for $0 <
\delta \le \delta_{0}$, that \eqref{eq:full_pd} implies
\begin{alignat*}{4}
    &E_{\xi}(\delta) \, &\le \, &\dfrac{\delta}{\pi} 
        \int_{0}^{k_0(\delta)} 
        \dfrac{(d_1-d_0)\left\|\rho\right\|^2_{L^2(\mathcal{M})} 
        \ee^{-2kd_0}\ee^{4ka}\left(2\delta^2+\lambda\delta^{\beta+1}
        \right)^2}{9k\delta^2} \ee^{2ka}L \dd k \\*
    &&&+ \dfrac{\delta}{\pi} \int_{k_0(\delta)}^{\infty} 
        \dfrac{(d_1-d_0)\left\|\rho\right\|^2_{L^2(\mathcal{M})} 
        \ee^{-2kd_0} \ee^{ka} 
        \left(2\delta^2+\lambda\delta^{\beta+1}\right)^{\frac{1}{2}}}
        {k\delta^2} \ee^{2ka}L \dd k 
        \\
    &\,&= \, &C_5\delta \int_{0}^{k_0(\delta)} 
        \dfrac{\ee^{-2k(d_0-3a)}}{k}
        \left(2\delta+\lambda\delta^{\beta}\right)^2 L \dd k \\
    &&&+ 9C_5\delta^{-\frac{1}{2}} \int_{k_0(\delta)}^{\infty} 
        \dfrac{\ee^{-2k(d_0-\frac{3}{2}a)}}{k} \left(2\delta
        +\lambda\delta^{\beta}\right)^{\frac{1}{2}} 
        L\dd k,
\end{alignat*}
where 
\begin{equation*}
    C_5 \equiv \dfrac{(d_1-d_0)
        \left\|\rho\right\|^2_{L^2(\mathcal{M})}}{9\pi}.
\end{equation*}
Using \eqref{eq:L} we can rewrite the above upper bound as
\begin{equation}\label{eq:lb_pd_T}
    E_{\xi}(\delta) \le T_1 + T_2 + T_3 + T_4, 
\end{equation}
where 
\begin{subequations}\label{eq:Ts}
    \begin{align}
        T_1 &\equiv 
            C_5\delta(2\delta+\lambda\delta^{\beta})^2
            \int_{0}^{k_0(\delta)}\ee^{-2k(d_0-3a)}
            \left(\dfrac{1-\ee^{-2k\xi}}{k}\right) \dd k;
            \label{eq:T_1}\\
        T_2 &\equiv 
            C_5\delta(\lambda^2\delta^{2\beta}+4)\int_{0}^{k_0(\delta)}
            \ee^{-2k(d_0-3a)}\ee^{-4ka}
            \left(\dfrac{\ee^{2k\xi}-1}{k}\right) \dd k; 
            \label{eq:T_2}\\
        T_3 &\equiv 
            9C_5\delta^{-\frac{1}{2}}(2\delta+\lambda
            \delta^{\beta})^{\frac{1}{2}}\int_{k_0(\delta)}^{\infty} 
            \ee^{-2k(d_0-\frac{3}{2}a)}
            \left(\dfrac{1-\ee^{-2k\xi}}{k}\right)\dd k;\label{eq:T_3}\\
        T_4 &\equiv 
            9C_5\delta^{-\frac{1}{2}}(2\delta+\lambda
            \delta^{\beta})^{-\frac{3}{2}}(\lambda^2\delta^{2\beta}+4)
            \int_{k_0(\delta)}^{\infty}\ee^{-2k(d_0-\frac{3}{2}a)}
            \ee^{-4ka}\left(\dfrac{\ee^{2k\xi}-1}{k}\right)\dd k.
            \label{eq:T_4}
    \end{align}
\end{subequations}
We derive estimates of these integrals in the next four lemmas. Recall
that $0 < \delta_0 \le \delta_{\mu}$ is such that $k_0(\delta) > 0$ for
$0 < \delta \le \delta_0$; we will assume $0 < \delta \le \delta_0$ for
the remainder of this section.
\begin{lemma}\label{lem:T_1}
    Suppose $\beta > 0$, $\lambda$ is feasible, $0 < \xi < a$, and $d_0
    \ge \tau(\beta)a$. Then 
    \begin{equation*}
        \ds\lim_{\delta \rightarrow 0^+} T_1 = 
        \begin{cases}
            C_6\lambda^{[2+(d_0-3a)/a]} &\text{if} \ 0 < \beta < 1 
                \ \text{and} \ d_0 = \tau(\beta)a, \\
            C_6(2+\lambda)^{[2+(d_0-3a)/a]} &\text{if} \ \beta = 1 
                \ \text{and} \ d_0 = \tau(\beta)a, \\
            C_62^{[2+(d_0-3a)/a]} &\text{if} \ \beta > 1 
                \ \text{and} \ d_0 = \tau(\beta)a, \\
            0 &\text{if} \ d_0 > \tau(\beta)a,
        \end{cases}
    \end{equation*}
    where
    \begin{equation*}
        C_6 = \xi C_5(d_0-3a)^{-1}.
    \end{equation*}
\end{lemma}
\begin{proof}
    We begin by noting that \eqref{eq:tau} implies that $(3/2)a \le
    \tau(\beta)a < 2a$ for all $\beta > 0$. Next, the function
    $k^{-1}(1-\ee^{-2k\xi})$ tends to 0 as $k$ goes to infinity and is
    continuous and decreasing for $k \in [0,\infty)$ as long as we
    define it to be equal to $2\xi$ at $k = 0$. Thus
    $k^{-1}(1-\ee^{-2k\xi}) \le 2\xi$ for all $k\ge0$. If $d_0 \ne 3a$,
    then this implies
    \begin{align}
        T_1
        &\le 2\xi C_5\delta(2\delta+\lambda\delta^{\beta})^2
            \int_{0}^{k_0(\delta)}\ee^{-2k(d_0-3a)} \dd k 
            \label{eq:T1_first_intermediate}\\
        &= \dfrac{2\xi C_5}{2(d_0-3a)}
            \delta(2\delta+\lambda\delta^{\beta})^2
            \left[1-\ee^{-2k_0(\delta)(d_0-3a)}\right] \nonumber \\
        &= C_6\delta(2\delta+\lambda\delta^{\beta})^2 
            - C_6\delta(2\delta+\lambda\delta^{\beta})^2
            \ee^{-2k_0(\delta)(d_0-3a)}. \label{eq:T1_intermediate}
    \end{align}
    The first term in \eqref{eq:T1_intermediate} goes to 0 as $\delta
    \rightarrow 0^+$. The second term is equal to
    \begin{equation}\label{eq:T1_second_term}
        C_6\delta(2\delta+\lambda\delta^{\beta})^2(2\delta^2
            +\lambda\delta^{\beta+1})^{(d_0-3a)/a}.
    \end{equation}
    
    If $0 < \beta < 1$ we rewrite this as
    \[
        C_6(2\delta^{1-\beta}+\lambda)^2(2\delta^{1-\beta}
            +\lambda)^{(d_0-3a)/a}
            \delta^{\left[1+2\beta+(\beta+1)(d_0-3a)/a\right]}. 
    \]
    This expression goes to 0 as $\delta\rightarrow 0^+$ if and only if
    \[ 
        1+2\beta+(\beta+1)\left(\frac{d_0-3a}{a}\right) > 0 
            \Leftrightarrow d_0 > \left(\dfrac{\beta+2}{\beta+1}\right)a
            = \tau(\beta)a,
    \]
    and it goes to $C_6\lambda^{[2+(d_0-3a)/a]}$ as $\delta\rightarrow
    0^+$ if and only if $d_0 = \tau(\beta)a$.
    
    If $\beta \ge 1$ we rewrite \eqref{eq:T1_second_term} as
    \[
        C_6(2+\lambda\delta^{\beta-1})^2
            (2+\lambda\delta^{\beta-1})^{(d_0-3a)/a}
            \delta^{[3+2(d_0-3a)/a]}.
    \]
    This term goes to 0 as $\delta\rightarrow 0^+$ if and only if
    \[
        3+2(d_0-3a)/a > 0\Leftrightarrow d_0 >\frac{3}{2}a=\tau(\beta)a,
    \]
    and if $d_0 = \tau(\beta)a$ it goes to $C_62^{[2+(d_0-3a)/a]}$ if
    $\beta > 1$ and $C_6(2+\lambda)^{[2+(d_0-3a)/a]}$ if $\beta = 1$.  

    If $d_0 = 3a$, then from \eqref{eq:T1_first_intermediate} we have
    \[
        T_1 \le 2\xi C_5
            \delta(2\delta+\lambda\delta^{\beta})^2k_0(\delta) 
            = a^{-1}\xi C_5\delta(2\delta+\lambda\delta^{\beta})^2
            \ln\left(\dfrac{1}
            {2\delta^2+\lambda\delta^{\beta+1}}\right);
    \]
    this expression goes to $0$ as $\delta \rightarrow 0^+$ for all 
    $\beta > 0$.
\end{proof}
\begin{lemma}\label{lem:T_2}
    Suppose $\beta > 0$, $\lambda$ is feasible, $0 < \xi < a/2$,
    and $d_0 \ge \tau(\beta)a$. Then 
    \begin{equation*}
        \ds\lim_{\delta \rightarrow 0^+} T_2=0.
    \end{equation*}
\end{lemma}
\begin{proof}
    We begin by noting that the function $k^{-1}(\ee^{2k\xi}-1)$ is
    continuous for $k \in [0,\infty)$ if we define it to be equal to
    $2\xi$ at $k = 0$. Also, since $d_0 \ge \tau(\beta)a \ge
    (3/2)a$, we have $\ee^{-2k(d_0-3a)}\ee^{-4ka} \le \ee^{-ka}$
    for all $k \ge 0$. This implies the integral
    \[
        \ds\int_{0}^{\infty} \ee^{-2k(d_0-3a)}\ee^{-4ka}
            \left(\dfrac{\ee^{2k\xi}-1}{k}\right) \dd k 
    \]
    converges to a positive constant $C$ as long as $0 < \xi <
    a/2$. Then \eqref{eq:T_2} implies that 
    \[
        T_2 \le CC_5\delta(\lambda^2\delta^{2\beta}+4) \rightarrow 0
        \quad \text{as} \ \delta \rightarrow 0^+.
    \]  
\end{proof}
\begin{lemma}\label{lem:T_3}
    Suppose $\beta > 0$, $\lambda$ is feasible, $0 < \xi < a$, and $d_0
    > (3/2)a$. Then
    \begin{equation*}
        \ds\lim_{\delta \rightarrow 0^+} T_3 = 
        \begin{cases}
            C_7\lambda^{[\frac{1}{2}+(d_0-\frac{3}{2}a)/a]} 
                &\text{if} \ 0 < \beta < 1 \ \text{and}\ d_0 
                = \tau(\beta)a, \\
            0 &\text{if} \ d_0 > \tau(\beta)a,
        \end{cases}
    \end{equation*}
    where
    \begin{equation*}
        C_7 = \dfrac{9C_5\xi}{d_0-\frac{3}{2}a} > 0.
    \end{equation*}
\end{lemma}
\begin{proof}
    As in the proof of Lemma~\ref{lem:T_1} we have
    $k^{-1}(1-\ee^{-2k\xi}) \le 2\xi$ for all $k \ge 0$. Thus
    \eqref{eq:T_3} implies
    \begin{align}
        T_3 
        &\le 18C_5\xi\delta^{-\frac{1}{2}}(2\delta
            +\lambda\delta^{\beta})^{\frac{1}{2}}
            \int_{k_0(\delta)}^{\infty} 
            \ee^{-2k(d_0-\frac{3}{2}a)} \dd k \nonumber \\
        &= \dfrac{18C_5\xi}{2(d_0-\frac{3}{2}a)}\delta^{-\frac{1}{2}}
            (2\delta+\lambda\delta^{\beta})^{\frac{1}{2}}
            \left[-\left.\ee^{-2k(d_0-\frac{3}{2}a)}
            \right|_{k_0(\delta)}^{\infty}\right] \nonumber \\
        &= C_7\delta^{-\frac{1}{2}}
            (2\delta+\lambda\delta^{\beta})^{\frac{1}{2}}
            \ee^{-2k_0(\delta)(d_0-\frac{3}{2}a)} \nonumber \\
        &= C_7\delta^{-\frac{1}{2}}
            (2\delta+\lambda\delta^{\beta})^{\frac{1}{2}}
            (2\delta^2+\lambda\delta^{\beta+1})^{(d_0-\frac{3}{2}a)/a}. 
            \label{eq:T3_intermediate}
    \end{align}
    
    If $0 < \beta < 1$, note that $\tau(\beta)a > (3/2)a$ --- this
    implies that the above analysis holds as long as $d_0 \ge
    \tau(\beta)a$. We rewrite \eqref{eq:T3_intermediate} as
    \[
        C_7(2\delta^{1-\beta}+\lambda)^{\frac{1}{2}}
        (2\delta^{1-\beta}+\lambda)^{(d_0-\frac{3}{2}a)/a}
        \delta^{[-\frac{1}{2}+\frac{\beta}{2}+(\beta+1)
        (d_0-\frac{3}{2}a)/a]}.
    \]
    This expression will go to 0 as $\delta\rightarrow 0^+$ if and only
    if
    \[
        -\dfrac{1}{2}+\dfrac{\beta}{2}+(\beta+1)
        \left(\dfrac{d_0-\frac{3}{2}a}{a}\right) > 0 
        \Leftrightarrow d_0 > \tau(\beta)a,
    \]
    and if $d_0 = \tau(\beta)a$ it goes to
    $C_7\lambda^{\{1/2+[d_0-(3/2)a]/a\}}$ as
    $\delta\rightarrow 0^+$.

    If $\beta \ge 1$ we note that the analysis leading to
    \eqref{eq:T3_intermediate} can only be applied if $d_0 >
    \tau(\beta)a = (3/2)a$. In this case we rewrite
    \eqref{eq:T3_intermediate} as 
    \[
        C_7(2+\lambda\delta^{\beta-1})^{\frac{1}{2}}
        (2+\lambda\delta^{\beta-1})^{(d_0-\frac{3}{2}a)/a}
        \delta^{2(d_0-\frac{3}{2}a)/a},
    \]
    which goes to 0 as $\delta\rightarrow 0^+$ if and only if
    $2[d_0-(3/2)a]/a > 0 \Leftrightarrow d_0 > \tau(\beta)a =
    (3/2)a$.  
\end{proof}
\begin{lemma}\label{lem:T_4}
    Suppose $\beta > 0$, $\lambda$ is feasible, $0 < \xi < a$, and $d_0
    \ge \tau(\beta)a$. Then
    \begin{equation*}
        	\lim_{\delta \rightarrow 0^+} T_4=0.
    \end{equation*}
\end{lemma}
\begin{proof}
    We have, from \eqref{eq:T_4}, that
    \begin{align}
        T_4 
        &= 9C_5\delta^{-\frac{1}{2}}(2\delta+\lambda\delta^{\beta})
            ^{-\frac{3}{2}}(\lambda^2\delta^{2\beta}+4)
            \int_{k_0(\delta)}^{\infty} 
            \ee^{-2k(d_0-\frac{3}{2}a)}\ee^{-4ka}
            \left(\dfrac{\ee^{2k\xi}-1}{k}\right)\dd k \nonumber \\
        &= 9C_5\delta^{-\frac{1}{2}}(2\delta+\lambda\delta^{\beta})
            ^{-\frac{3}{2}}(\lambda^2\delta^{2\beta}+4)
            \int_{k_0(\delta)}^{\infty} \ee^{-k(2d_0+a)}
            \left(\dfrac{\ee^{2k\xi}-1}{k}\right)\dd k \nonumber \\
        &\le 9C_5\dfrac{\delta^{-\frac{1}{2}}
            (2\delta+\lambda\delta^{\beta})^{-\frac{3}{2}}
            (\lambda^2\delta^{2\beta}+4)}{k_0(\delta)}
            \int_{k_0(\delta)}^{\infty} 
            \ee^{-k(2d_0+a-2\xi)}\dd k \nonumber \\
        &= \left[\dfrac{9C_5(\lambda^2\delta^{2\beta}+4)}
            {2d_0+a-2\xi}\right]
            \left[\dfrac{\delta^{-\frac{1}{2}}
            (2\delta+\lambda\delta^{\beta})
            ^{-\frac{3}{2}}}{k_0(\delta)}\right]
            \ee^{-k_0(\delta)(2d_0+a-2\xi)} 
            \nonumber \\
        &= C_8(\lambda^2\delta^{2\beta}+4)
            \left[\dfrac{\delta^{-\frac{1}{2}}
            (2\delta+\lambda\delta^{\beta})
            ^{-\frac{3}{2}}}{k_0(\delta)}\right](2\delta^2+\lambda
            \delta^{\beta+1})^{(2d_0+a-2\xi)/(2a)}, 
            \label{eq:T4_intermediate}
    \end{align}
    where 
    \begin{equation*}
        C_8 \equiv \dfrac{9C_5}{2d_0+a-2\xi} > 0.
    \end{equation*}
    
    If $0 < \beta < 1$ we rewrite \eqref{eq:T4_intermediate} as
    \[
        \left[\dfrac{C_8(\lambda^2\delta^{2\beta}+4)}
            {k_0(\delta)}\right]
            (2\delta^{1-\beta}+\lambda)
            ^{-\frac{3}{2}}(2\delta^{1-\beta}+\lambda)
            ^{(2d_0+a-2\xi)/(2a)}
            \delta^{[-\frac{1}{2}-\frac{3}{2}\beta 
            + (\beta+1)(2d_0+a-2\xi)/(2a)]}.
    \]
    This expression will go to 0 as $\delta\rightarrow 0^+$ if and only
    if
    \[
        -\frac{1}{2}-\frac{3}{2}\beta + 
            \frac{(\beta+1)(2d_0+a-2\xi)}{2a} \ge 0 
            \Leftrightarrow d_0 \ge 
            \left(\frac{\beta}{\beta+1}\right)a + \xi.
    \]
    We note that $\left[\beta/(\beta+1)\right]a + \xi <
    \tau(\beta)a$ since $0 < \beta < 1$ and $\xi < a$. Thus if $0 <
    \beta < 1$ and $d_0 \ge \tau(\beta)a$ we have $T_4 \rightarrow 0$ as
    $\delta \rightarrow 0^+$.  
    
    If $\beta \ge 1$ we rewrite \eqref{eq:T4_intermediate} as
    \[
        \left[\dfrac{C_8(\lambda^2\delta^{2\beta}+4)}
            {k_0(\delta)}\right](2+\lambda\delta^{\beta-1})
            ^{-\frac{3}{2}}
            (2+\lambda\delta^{\beta-1})^{(2d_0+a-2\xi)/(2a)}
            \delta^{[-2+(2d_0+a-2\xi)/a]}.
    \]
    This expression goes to 0 as $\delta\rightarrow 0^+$ if and only if
    \[
        -2+(2d_0+a-2\xi)/a \ge 0 \Leftrightarrow d_0 
            \ge \frac{a}{2} + \xi.
    \]
    Since $\beta \ge 1$ and $0 < \xi < a$ we have $a/2 + \xi <
    (3/2)a = \tau(\beta)a$; thus if $\beta \ge 1$ and $d_0 \ge
    \tau(\beta)a$ we have $T_4 \rightarrow 0$ as $\delta \rightarrow
    0^+$.
\end{proof}
We summarize our result from this section in the following theorem.
\begin{theorem}\label{thm:bounded}
    Let $\beta > 0$ and $\lambda$ feasible be fixed. Suppose also that
    $0 < \xi < a/2$ and $\rho \in \Pset$. If $d_0 >
    \tau(\beta)a$, then $\lim_{\delta \rightarrow 0^+} E_{\xi}(\delta) =
    0$.
\end{theorem}
\begin{proof}
    If the hypotheses of the theorem hold and if $\delta \le \delta_0$,
    then \eqref{eq:lb_pd_T} and
    Lemmas~\ref{lem:g2bounded}--\ref{lem:T_4} imply
    \begin{equation*}
        0 \le E_{\xi}(\delta) \le T_1 + T_2 + T_3 + T_4 \rightarrow 0 
            \quad \text{as}\ \delta \rightarrow 0^+.
    \end{equation*}
\end{proof}

Figures~\ref{fig:rectangle_bounded}~and~\ref{fig:circle_bounded} are
supporting numerical plots; they are the same as
Figures~\ref{fig:rectangle_blow_up}~and~\ref{fig:circle_blow_up},
respectively, except in this case we have taken $a = d_0/\tau(\beta)$ so
$\rho$ just touches the region of influence (in order to accomplish this
we have taken $\beta = 0.5$ in
Figures~\ref{fig:rectangle_bounded}(c)~and~\ref{fig:circle_bounded}(c)
rather than $\beta = 0.8$ as in
Figures~\ref{fig:rectangle_blow_up}(c)~and~\ref{fig:circle_blow_up}(c)).  
\begin{figure}[!hb]
	\begin{center}
    	\includegraphics{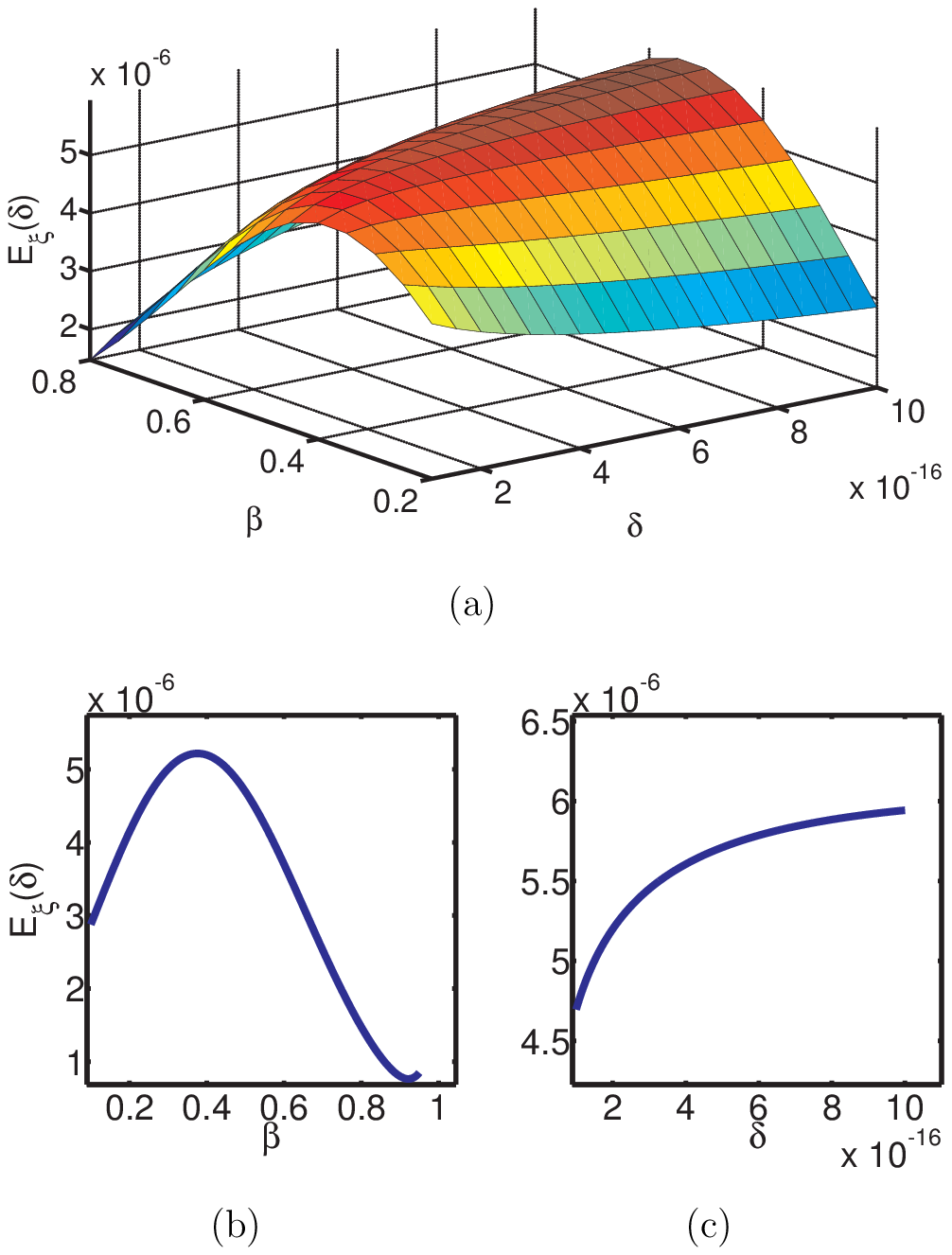}
    \end{center}
    \caption[Rectangular Charge Density Bounded]%
    {(Rectangular $\rho$) In all of these subfigures we take $a =
    d_0/\tau(\beta)$ so $\rho$ is completely outside the region of
    influence. (a) A plot of $E_{\xi}(\delta)$ versus $\beta$ and
    $\delta$ --- the $z$\nobreakdash -axis scale is $10^{-6}$; (b) a
    plot of $E_{\xi}(\delta)$ for $\delta = 10^{-16}$ as a function of
    $\beta$ --- the $y$\nobreakdash -axis scale is $10^{-6}$; (c) a plot
    of $E_{\xi}(\delta)$ for $\beta = 0.5$ as a function of $\delta$ ---
    the $y$\nobreakdash -axis scale is $10^{-6}$.}%
	\label{fig:rectangle_bounded}
\end{figure}
\begin{figure}[!bt]
	\begin{center}
    	\includegraphics{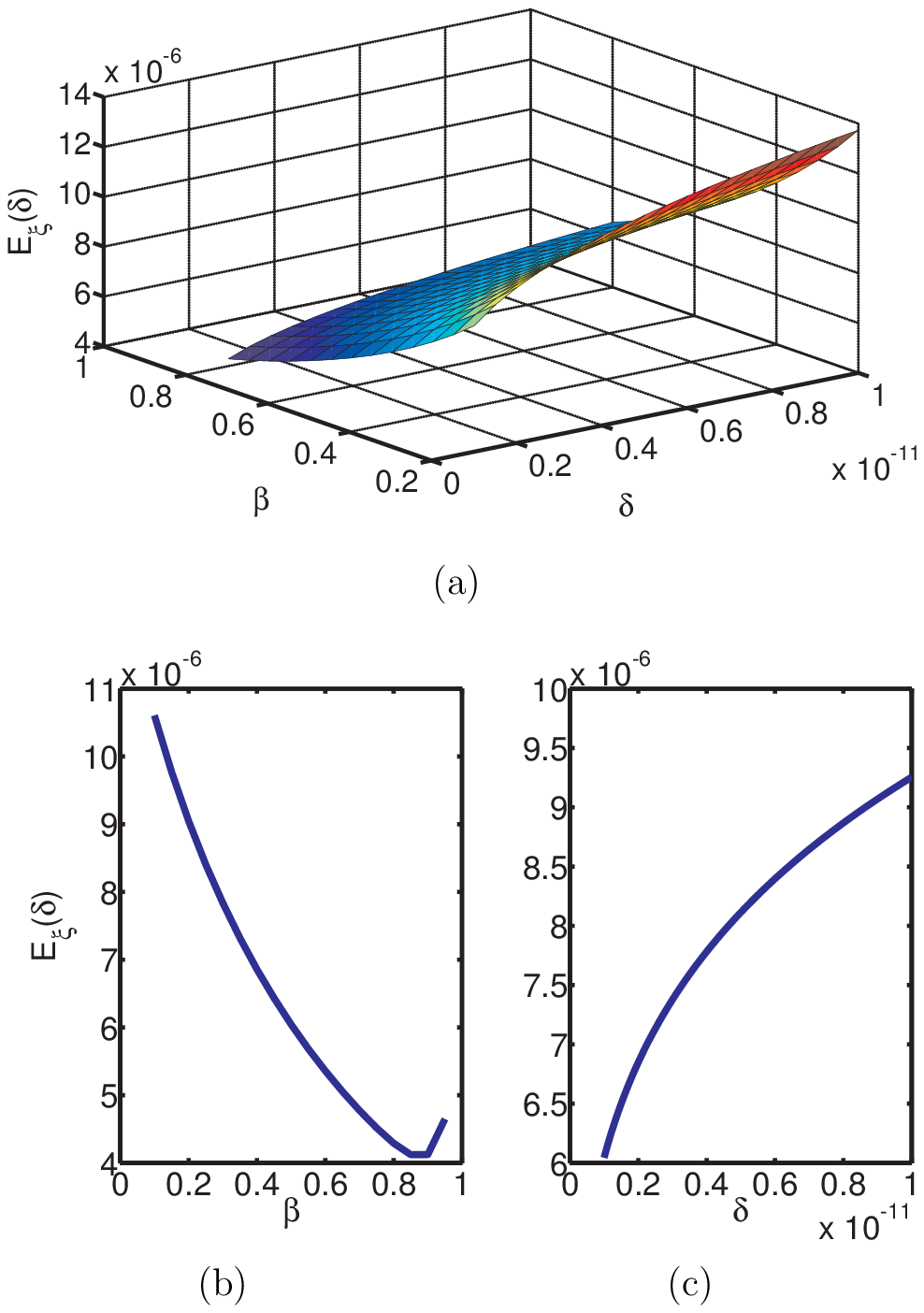}
    \end{center}
    \caption[Circular Charge Density Bounded]
    {(Circular $\rho$) In all of these subfigures we take $a =
    d_0/\tau(\beta)$ so $\rho$ is completely outside the region of
    influence. (a) A plot of $E_{\xi}(\delta)$ versus $\beta$ and
    $\delta$ --- the $z$\nobreakdash -axis scale is $10^{-5}$; (b) a
    plot of $E_{\xi}(\delta)$ for $\delta = 10^{-12}$ as a function of
    $\beta$ --- the $y$\nobreakdash -axis scale is $10^{-6}$; (c) a plot
    of $E_{\xi}(\delta)$ for $\beta = 0.5$ as a function of $\delta$ ---
    the $y$\nobreakdash -axis scale is $10^{-6}$.}%
	\label{fig:circle_bounded}
\end{figure}

\newpage
\section{Boundedness of the Potential}\label{sec:bounded_potential}

In this section, we derive bounds on the potential in regions far away
from the slab. In particular, we prove that the potentials $V_c$ and
$V_m$ to the left and right of the slab, respectively, are bounded by
constants that are independent of $\delta$ (for $|x|$ large enough). As
discussed in the Introduction, this is the second requirement for
cloaking by anomalous localized resonance to occur. At this point we do
not address questions regarding which portions of the (rescaled) charge
distribution $\rho/\sqrt{E_{\xi}(\delta)}$ will be cloaked. For example,
if the (rescaled) rectangular charge distribution from
Section~\ref{subsubsec:rectangle} is halfway inside the cloaking region
(so $x_0 = \tau(\beta)a$), we have not yet determined whether it will be
completely cloaked or if only the leading half will be cloaked.  

We begin with some some technical results. The proofs of the next two
lemmas are straightforward and can be found in work by one of the
authors of this paper \cite{Thaler:2014:BVI}.
\begin{lemma}\label{lem:plus}
    Let $\psi_k^+$ and $\psi_k^-$ be defined as in \eqref{eq:psiplus}
    and \eqref{eq:psiminus}, respectively. Then for each $k \in
    \mathbb{R}$ and all $0 < \delta \le \delta_{\mu}$, 
    \begin{equation*}
            ||k|\psi_k^+ + \psi_k^-|^2 \ge 2|k|^2\ee^{-2|k|a}.
    \end{equation*}
\end{lemma}
\begin{lemma}\label{lem:minus}
    Let $\psi_k^+$ and $\psi_k^-$ be defined as in \eqref{eq:psiplus}
    and \eqref{eq:psiminus}, respectively. Then there exists $0 <
    \delta_{\psi^-}(\beta,\lambda) \le \delta_{\mu}$ such that 
    \begin{equation*}
            \left|\psi_k^+ - \frac{1}{|k|}\psi_k^-\right|^2 
            \le \frac{5}{2}(\delta+\mu)^2\ee^{2|k|a}
    \end{equation*}    
    for all $k \in \mathbb{R}$ and all $0 < \delta \le \delta_{\psi^-}$.  
\end{lemma}


\subsection{The Potential $V_c$}\label{subsec:V_c_bounded}

Note that $V_c$ is harmonic for $x < 0$ by \eqref{eq:V_pde_full} and
\eqref{eq:dc}. In addition, since $V \in
L^2_{\mathrm{loc}}(\mathbb{R}^2)$, $V \in
L^1_{\mathrm{loc}}(\mathbb{R}^2)$ as well. Hence the Weyl Theorem (see,
e.g., Theorem~18.G in the book by Zeidler \cite{Zeidler:1990:NFA})
implies that $V$ is infinitely differentiable for $x < 0$ (after
modification on a set of measure $0$), so we can examine pointwise
values of $V_c$. The next lemma states that, far enough away from the
slab, the potential $V_c$ is bounded for all $\delta \le \delta_{\mu}$.
\begin{lemma}\label{lem:V_c_bounded}
    Suppose $\rho \in \Pset$. Then there is a positive constant $C_9$
    such that $V_c(x,y) \le C_9$ for all $x < -3a$ and for all $0 <
    \delta \le \delta_{\mu}$.
\end{lemma}
\begin{proof}
    From \eqref{eq:Vc_hat} and \eqref{eq:Ak} we have
    \begin{equation}\label{eq:Ak_bound_1}
        |\Vhat_c(x,k)|^2 = |A_k|^2\ee^{2|k|x} 
            = \frac{|I_k|^2\ee^{2|k|x}}{\ee^{-2|k|a}
            ||k|\psi_k^+ + \psi_k^-|^2}.
    \end{equation}
    In combination with Lemma~\ref{lem:plus}, this implies that
    \begin{equation}\label{eq:Ak_bound}
        |\Vhat_c(x,k)|^2 \le \frac{|I_k|^2}{2|k|^2}\ee^{2|k|(x+2a)}
    \end{equation}
    for $x < 0$, for all $k \in \mathbb{R}$, and for all $0 < \delta \le
    \delta_{\mu}$. In particular, note that the expression in
    \eqref{eq:Ak_bound} is an even function of $k$ if $\rho$ is
    real-valued due to Lemma~\ref{lem:Ik_properties}. Then for $x < 0$
    \eqref{eq:Ak_bound} implies that
    \begin{align}
        \ds\int_{-\infty}^{\infty} |\Vhat_c(x,k)|^2 \dd k
        &\le \frac{1}{2}\ds\int_{-\infty}^{\infty} 
            \frac{|I_k|^2}{|k|^2}\ee^{2|k|(x+2a)} \dd k \nonumber\\
        &= \int_{0}^{\infty} 
            \frac{|I_k|^2}{|k|^2}\ee^{2|k|(x+2a)} \dd k \nonumber \\
        &= \int_{0}^{1} 
            \frac{|I_k|^2}{k^2}\ee^{2k(x+2a)} \dd k
            + \int_{1}^{\infty} 
            \frac{|I_k|^2}{k^2}\ee^{2k(x+2a)} \dd k\nonumber \\
        &= \int_{0}^{1} 
            \frac{|I_k|^2}{k^2}\ee^{2k(x+2a)} \dd k
            + (d_1-d_0)\|\rho\|^2_{L^2(\mathcal{M})} 
            \int_{1}^{\infty} 
            \frac{\ee^{2k(x+2a-d_0)}}{k^2}\dd k,
            \label{eq:split_integral}
    \end{align}
    thanks to Lemma~\ref{lem:Ik_properties}. Since
    \[
        \frac{|I_k|^2}{k^2} \le C_I^2
    \] 
    for $k \ge 0$ by Lemma~\ref{lem:Ik_properties}, the first integral
    in \eqref{eq:split_integral} converges for any $x \in \mathbb{R}$.
    The second integral in \eqref{eq:split_integral} converges if and
    only if $x \le d_0-2a$ (note that $d_0 -2a > -a$ since $d_0 > a$).
    Then if $x < -2a$ we have, from \eqref{eq:split_integral}, that
    \begin{equation*}
        \int_{-\infty}^{\infty} |\Vhat_c(x,k)|^2 \dd k\le 
        \int_{0}^{1} C_I^2 \dd k
            + (d_1-d_0)\|\rho\|^2_{L^2(\mathcal{M})} 
            \int_{1}^{\infty} \frac{1}{k^2}\dd k 
        = C_I^2 + (d_1-d_0)\|\rho\|^2_{L^2(\mathcal{M})}.     
    \end{equation*}

    Then the Plancherel Theorem \eqref{eq:Plancherel} implies that for
    each $x < -2a$ we have
    \begin{equation}\label{V_c_intermediate_bound}
        \int_{-\infty}^{\infty} |V_c(x,y)|^2 \dd y 
            = \frac{1}{2\pi}\int_{-\infty}^{\infty}|\Vhat_c(x,k)|^2\dd k
            \le \frac{1}{2\pi}
            \left[C_I^2 + (d_1-d_0)\|\rho\|^2_{L^2(\mathcal{M})}\right].
    \end{equation}
    Since $V_c(x,y)$ is harmonic for $x < -2a$, it satisfies the mean
    value property (see, e.g., Chapter 2 in the book by Evans
    \cite{Evans:2010:PDE}): for any point $(x,y)$ with $x < -3a$ we have
    \[
        V(x,y) = \frac{1}{|B_a((x,y))|}\int_{B_a((x,y))}V(x',y') 
            \dd y' \dd x',
    \]
    where $B_a((x,y))$ is the ball of radius $a$ centered at the point
    $(x,y)$; note that all points $(x',y') \in B_a((x,y))$ satisfy $x' <
    -2a$ since $x < -3a$. Finally by the Cauchy--Schwarz inequality and
    \eqref{V_c_intermediate_bound} we have
    \begin{align*}
        |V_c(x,y)| 
        &= \frac{1}{|B_a((x,y))|}
            \left|\int_{B_a((x,y))}V(x',y') \dd y'\dd x'\right| \\
        &\le \frac{1}{|B_a((x,y))|}\int_{B_a((x,y))}|V(x',y')|
            \dd y'\dd x' \\
        &\le \frac{1}{|B_a((x,y))|} \left[\int_{B_a((x,y))}|V(x',y')|^2
            \dd y' \dd x'\right]^{\frac{1}{2}}
            \left[\int_{B_a((x,y))} \dd y' \dd x'\right]^{\frac{1}{2}}\\
        &\le \frac{1}{|B_a((x,y))|^{\frac{1}{2}}} 
            \left[\int_{x-a}^{x+a}\int_{-\infty}^{\infty} |V(x',y')|^2 
            \dd y'\dd x'\right]^{\frac{1}{2}} \\
        &\le \int_{x-a}^{x+a} \frac{1}{2\pi^{3/2}a}
            \left[C_I^2 + (d_1-d_0)\|\rho\|^2_{L^2(\mathcal{M})}\right] 
            \dd x' \\
        &= C_9,
    \end{align*}
    where $C_9 = \pi^{-3/2}\left[C_I^2 
    + (d_1-d_0)\|\rho\|^2_{L^2(\mathcal{M})}\right]$.
\end{proof}


\subsection{The Potential $V_m$}\label{subsec:V_m_bounded}

We will now show that $|V_m(x,y)|$ is bounded for $x$ large enough. In
particular, we at least assume that $x > d_1$. We begin with a lemma
that is very similar to Lemma~\ref{lem:Ik_properties}. For $x > d_1$ we
define 
\begin{equation}\label{eq:Jk}
        J_k(x) \equiv \int_{d_0}^{d_1}\rhohat(s,k)\ee^{-|k|(x-s)}\dd s.
\end{equation}
The proof of the following lemma can be found in \cite{Thaler:2014:BVI}.
\begin{lemma}\label{lem:Jk_properties}
    Suppose $\rho \in \Pset$ (where $\Pset$ is defined in
    \eqref{eq:P_def}) and that, for $x > d_1$, $J_k(x)$ is defined as in
    \eqref{eq:Jk}. Then, for every $x > d_1$, $J_k(x)$ satisfies the
    following properties:  
    \begin{enumerate}
        \item for all $k \in \mathbb{R}$, $|J_k(x)|^2 \le
            (d_1-d_0)\|\rho\|^2_{L^2(\mathcal{M})}\ee^{-2k(x-d_1)}$;
        \item if $\rho$ is real-valued, then $|J_k(x)|^2$ is an even 
            function of $k$ for $k \in \mathbb{R}$;
        \item $J_k(x)$ is continuous at $k$ for each $k \in \mathbb{R}$;
        \item $\ds\lim_{k\rightarrow 0} J_k(x) = J_0(x) = 0$;
        \item for each $x > d_1$, 
            \[
                \lim_{k\rightarrow 0}\dfrac{|J_k(x)|}{|k|} = |C_{0}| <
                    \infty, 
            \]
            where $C_0$ is defined in
            Lemma~\ref{lem:Ik_properties}; moreover, there is a positive
            constant $C_J$, independent of $x$, such that $|J_k(x)|/|k|
            \le C_J$ for all $x > d_1$ and all $k \in [0,1]$.  
   \end{enumerate}
\end{lemma}
\begin{lemma}\label{lem:V_m_bounded}
    Suppose $\rho \in \Pset$. Then there is a positive constant $C_{10}$
    such that $|V_m(x,y)| \le C_{10}$ for all $x > a + \max\{d_1,4a\}$
    and for all $\delta \le \delta_{\psi^-}$ (where $\delta_{\psi^-}$ is
    defined in Lemma~\ref{lem:minus}).
\end{lemma}
\begin{proof}
    Based on our choice of $A_k$ and $I_k$ in \eqref{eq:Ak} and
    \eqref{eq:Ik}, respectively, for $x > d_1$ we have
    \begin{equation}\label{eq:V_m_hat_again} 
        \Vhat_m(x,k) = \ee^{-|k|x}
            \left(\dfrac{A_k\psi_k^+\ee^{|k|a}}{2}-
            \dfrac{A_k\psi_k^-\ee^{|k|a}}{2|k|}\right)
            +\dfrac{J_k(x)}{2|k|};
    \end{equation}
    see \eqref{eq:Vm_hat}. Then \eqref{eq:Ak}, the triangle inequality,
    and the fact that $(p+q)^2 \le 2p^2+2q^2$ for real numbers $p$ and
    $q$ imply, for $x > d_1$, that
    \begin{align*}
        |\Vhat_m(x,k)|^2 
        &= \left|\ee^{-|k|x}\left(\dfrac{A_k\psi_k^+\ee^{|k|a}}{2}-
            \dfrac{A_k\psi_k^-\ee^{|k|a}}{2|k|}\right)
            +\dfrac{J_k(x)}{2|k|}\right|^2 \\
        &\le \frac{\ee^{-2|k|(x-a)}}{2}|A_k|^2
            \left|\psi_k^+-\frac{1}{|k|}\psi_k^-\right|^2 
            + \dfrac{|J_k(x)|^2}{2|k|^2}.
    \end{align*}
    Then \eqref{eq:Ak_bound_1}, Lemma~\ref{lem:plus}, and
    Lemma~\ref{lem:minus} imply, for $0 < \delta \le \delta_{\psi^-}$,
    that
    \begin{align*}
        |\Vhat_m(x,k)|^2 
        &\le 
        \frac{\ee^{-2|k|(x-a)}}{2}\frac{|I_k|^2\ee^{4|k|a}}{2|k|^2}
            \left|\psi_k^+-\frac{1}{|k|}\psi_k^-\right|^2 
            +\dfrac{|J_k(x)|^2}{2|k|^2} \\
        &\le \frac{5\ee^{-2|k|(x-3a)}|I_k|^2}{8|k|^2}
            (\delta+\mu)^2\ee^{2|k|a}+\dfrac{|J_k(x)|^2}{2|k|^2} \\
        &\le\frac{5}{8}(\delta+\mu)^2\frac{|I_k|^2}{|k|^2}
            \ee^{-2|k|(x-4a)}+\dfrac{|J_k(x)|^2}{2|k|^2}.
            \numberthis \label{eq:V_m_bounded_intermediate}
    \end{align*}
    Note that the expression in \eqref{eq:V_m_bounded_intermediate} is
    even as a function of $k$ by
    Lemmas~\ref{lem:Ik_properties}~and~\ref{lem:Jk_properties}. Then we
    have
    \begin{alignat*}{4}
        &\ds\int_{-\infty}^{\infty} |\Vhat_m(x,k)|^2 \dd k
        &\, \le \, &\frac{5}{8}(\delta+\mu)^2
            \int_{-\infty}^{\infty}\frac{|I_k|^2}{|k|^2}
            \ee^{-2|k|(x-4a)} \dd k 
            + \int_{-\infty}^{\infty}\dfrac{|J_k(x)|^2}{2|k|^2}
                \dd k \\
        &&\, = \, &\frac{5}{4}(\delta+\mu)^2\left[
            \int_{0}^{1}\frac{|I_k|^2}{k^2}
                \ee^{-2k(x-4a)} \dd k 
                + \int_{1}^{\infty}\frac{|I_k|^2}{k^2}
                \ee^{-2k(x-4a)} \dd k \right] \\*
        &&&+ \int_{0}^{1}\dfrac{|J_k(x)|^2}{k^2}
                \dd k + \int_{1}^{\infty}\dfrac{|J_k(x)|^2}{k^2}
                \dd k.
    \end{alignat*}
    Then Lemmas~\ref{lem:Ik_properties}~and~\ref{lem:Jk_properties}
    imply
    \begin{align*}
        &\ds\int_{-\infty}^{\infty} |\Vhat_m(x,k)|^2 \dd k 
        \le  \frac{5}{4}(\delta+\mu)^2C_0^2
            \int_0^1\ee^{-2k(x-4a)} \dd k + C_J^2 \\
        &\qquad +(d_1-d_0)\left\|\rho\right\|^2_{L^2(\mathcal{M})}
                \left[\frac{5}{4}(\delta+\mu)^2
                \int_{1}^{\infty}\frac{\ee^{-2k(x-4a+d_0)}}{k^2}\dd k 
                + \int_{1}^{\infty}\dfrac{\ee^{-2k(x-d_1)}}{k^2}\dd k
                \right]. \numberthis\label{eq:V_m_int}
    \end{align*}

    If $x > \max\{d_1,4a\}$, then all of the integrals in
    \eqref{eq:V_m_int} converge. In particular, the integral from $0$ to
    $1$ and both of the integrals from $1$ to $\infty$ converge to
    numbers less than or equal to $1$ in that case. Therefore
    \eqref{eq:V_m_int} becomes 
    \begin{equation*}
        \int_{-\infty}^{\infty} |\Vhat_m(x,k)|^2 \dd k
        \le \frac{5}{4}(\delta+\mu)^2C_0^2 + C_J^2
        + (d_1-d_0)\left\|\rho\right\|^2_{L^2(\mathcal{M})}
        \left[\frac{5}{4}(\delta+\mu)^2 + 1\right]\\
        \equiv \widetilde{C}_{10}.
    \end{equation*}
    If we define $b \equiv a+\max\{d_1,4a\}$, for example, then for $x >
    b$ each point $(x',y') \in B_a((x,y))$ satisfies $x' >
    \max\{d_1,4a\}$. Since $V_m$ is harmonic in the region where $x' >
    d_1$, it satisfies the mean value property there. Using this in
    combination with the Plancherel Theorem (just as in the proof of
    Lemma~\ref{lem:V_c_bounded}) gives
    \[
        |V_m(x,y)| \le \int_{x-a}^{x+a} 
            \frac{\widetilde{C}_{10}}{2\pi^{3/2}a} \dd x' \equiv C_{10},
    \]
    where $C_{10} = \pi^{-3/2}\widetilde{C}_{10}$.
\end{proof}


\section*{Acknowledgments}
AET would like to thank Patrick Bardsley, Elena Cherkaev, David Dobson,
Fernando Guevara Vasquez, and Hyeonbae Kang for helpful discussions.



\ifx \bblindex \undefined \def \bblindex #1{} \fi
\providecommand{\bysame}{\leavevmode\hbox to3em{\hrulefill}\thinspace}
\providecommand{\MR}{\relax\ifhmode\unskip\space\fi MR }
\providecommand{\MRhref}[2]{%
  \href{http://www.ams.org/mathscinet-getitem?mr=#1}{#2}
}
\providecommand{\href}[2]{#2}

\end{document}